\documentclass[a4paper,11pt]{amsart}

\usepackage[left=2.7cm,right=2.7cm,top=3.5cm,bottom=3cm]{geometry}
\usepackage{latexsym}
\usepackage{amsthm,amssymb,amsmath,amsfonts,amscd}
\usepackage{verbatim}
\usepackage[latin1]{inputenc}

 \usepackage{latexsym}
 \usepackage{longtable}
\newfont{\cyr}{wncyr10 scaled 1100}
\newfont{\cyrr}{wncyr9 scaled 1000}

\theoremstyle{plain}
\newtheorem{theorem}{Theorem}[section]

\newtheorem{lemma}[theorem]{Lemma}

\newtheorem{thm}{Theorem}[section]

\newtheorem{cor}[thm]{Corollary}

\theoremstyle{definition}

\newtheorem{definition}[theorem]{Definition}

\newtheorem{dfn}[thm]{Definition}

\theoremstyle{remark}
\newtheorem{remark}[theorem]{Remark}

\newtheorem{rmk}[thm]{Remark}

\usepackage[usenames]{color}
\definecolor{Indigo}{rgb}{0.2,0.1,0.7}
\definecolor{Violet}{rgb}{0.5,0.1,0.7}
\definecolor{White}{rgb}{1,1,1}
\definecolor{Green}{rgb}{0.1,0.9,0.2}

\newcommand{\mat}[4]{\left(\begin{array}{cc}#1&#2\\#3&#4\end{array}\right)}
\newcommand{\smallmat}[4]{\bigl(\begin{smallmatrix}#1&#2\\#3&#4\end{smallmatrix}\bigr)}

\newcommand{\p}{\mathfrak p}

\newcommand{\W}{\mathbb W}
\newcommand{\D}{\mathbb D}

\newcommand{\PP}{\mathbb P}

\newcommand{\pwseries}[1]{[[#1]]}

\newcommand{\arr}{{\; \longrightarrow \;}}

\newcommand{\Des}{{\operatorname{Des }}}

\newcommand{\Hom}{{\operatorname{Hom}}}

\newcommand{\ord}{{\operatorname{ord }}}
\newcommand{\cusp}{{\operatorname{cusp }}}
\newcommand{\new}{{\operatorname{new }}}

\newcommand{\Spec}{{\operatorname{Spec }}}

\newcommand{\SL}{{\operatorname{SL }}}
\newcommand{\GL}{{\operatorname{GL}}}

\newcommand{\U}{{\operatorname{U}}}

\newcommand{\Sel}{{\operatorname{Sel}}}

\newcommand{\gerV}{{\frak{V}}}

\newcommand{\calA}{{\mathcal{A}}}

\newcommand{\calE}{{\mathcal{E}}}

\newcommand{\calH}{{\mathcal{H}}}

\newcommand{\calJ}{{\mathcal{J}}}

\newcommand{\calL}{{\mathcal{L}}}
\newcommand{\calM}{{\mathcal{M}}}

\newcommand{\calO}{{\mathcal{O}}}

\newcommand{\calR}{{\mathcal{R}}}
\newcommand{\calS}{{\mathcal{S}}}

\newcommand{\calV}{{\mathcal{V}}}
\newcommand{\calW}{{\mathcal{W}}}
\newcommand{\calX}{{\mathcal{X}}}

\newcommand{\calZ}{{\mathcal{Z}}}

\def\C{\mathbb{C}}
\def\D{\mathbb{D}}

\def\G{\mathbb{G}}

\def\I{\mathbb{I}}
\def\J{\mathbb{J}}
\def\K{\mathbb{K}}
\def\L{\mathbb{L}}
\def\M{\mathbb{M}}
\def\N{\mathbb{N}}

\def\Q{\mathbb{Q}}

\def\T{\mathbb{T}}

\def\W{\mathbb{W}}

\def\Z{\mathbb{Z}}

\def\Fp{\overline{\mathbb{F}}_p}

\include{thebibliography}

\begin{document}

\title{$\Lambda$-adic Families of Jacobi Forms}
% \today 
\author{Matteo Longo, Marc-Hubert Nicole}

\begin{abstract} 

We show that Hida's families of $p$-adic elliptic modular forms generalize to $p$-adic families of Jacobi forms. We also construct $p$-adic versions of theta lifts from elliptic modular forms to Jacobi forms. Our results extend to Jacobi forms previous works by Hida and Stevens on the related case of half-integral weight modular forms.\end{abstract}

\subjclass[2000]{}
\maketitle

%\tableofcontents

 \section{Introduction}
 The theory of $p$-adic families of ordinary cusp forms was initiated by Hida (\cite{Hida85}, \cite{Hida86},  \cite{Hida88}), starting with the group $\GL(2)$. Over the last 30 years, its scope extended greatly, encompassing groups associated to Shimura data of Hodge type, and beyond. Historically though, some of the first extensions of the theory were outside the realm of reductive groups, for example the metaplectic group governing half-integral weights modular forms, worked out by Hida himself in \cite{HarHalf} by d\'evissage to his original theory. Another important theme of $p$-adic interpolation has been to establish variants in families of various functorial liftings between spaces of modular forms. For half-integral weight modular forms, the $p$-adic variant of the theta lift due originally to Shintani was studied by Stevens in \cite{St}.
 In this paper, we wish to substitute, in the non-reductive examples mentioned above, half-integral weight modular forms with Jacobi forms, and verify that the corresponding objects and liftings vary equally well in ordinary families, thus extending some results of Guerzhoy \cite{GuerzhoyCrelle}. 
 
In this paper we therefore have two aims: (1) Develop a theory of $p$-adic families of ordinary Jacobi forms, following the approach of Hida \cite{HarHalf}; (2) Construct, following \cite{St}, an example of theta lifts from $p$-adic families of ordinary elliptic modular forms to $p$-adic families of ordinary Jacobi forms which interpolates the Shimura-Shintani correspondence for classical forms. 

 Before explaining our results in more details, we stress that one of the motivations to develop such a theory, especially the $p$-adic theta correspondence, lies with companion paper \cite{LN-GKZ} which investigates an analogue for $p$-adic families of the Gross-Kohnen-Zagier theorem (GKZ for short), in which Heegner points are substituted with $p$-adic families of Heegner points called Big Heegner points. It seems that a satisfactory, although still conjectural,  generalisation of the GKZ Theorem for Big Heegner points involves the $p$-adic theta correspondence that we construct in this paper. In Section \ref{final sec} we describe in a more precise way the relation between this paper and \cite{LN-GKZ}.  

In the first part of the paper, we develop concisely the general framework for $p$-adic families of Jacobi forms, relying on unpublished work \cite{CandeloriTheta} applied only in the much simpler and much better studied case of dimension one. This part can be seen as an extension to Jacobi forms of the analogous theory developed by Hida in \cite{HarHalf} for half-integral weight modular forms. The main result of this first part is Corollary 
\ref{coro 2.14}, which is an analogue of the celebrated Hida Control Theorem (\cite[Theorem II]{Hida88}) in the context of $p$-adic families of ordinary Jacobi cusp forms. Corollary \ref{coro 2.14} shows that the $\Lambda$-module of $p$-adic families of Jacobi forms specializes to vector spaces of classical Jacobi forms when it is cut out by arithmetic primes. We now state a simplified form of the Control Theorem for Jacobi forms. 
Fix an integer $N\geq 1$ and a prime number $p\nmid N$. 
Let $\Lambda=\mathcal{O}[\![1+p\Z_p]\!]$ be the Iwasawa algebra of $1+p\Z_p$ with coefficients in the valuation ring $\mathcal{O}$ of a fixed finite extension of $\Q_p$. 
We consider a normal integral domain $\mathbb{I}$ which is finite over $\Lambda$. 
Put $\mathcal{X}(\mathbb{I})=\Spec(\mathbb{I})(\C_p)$. A point $\kappa\in\mathcal{X}(\mathbb{I})$ is said to be  {arithmetic} if its restriction to $1+p\Z_p\subseteq\mathbb{I}$ is of the form $\gamma\mapsto \psi(\gamma)\gamma^{k-2}$ for a finite order character  $\psi$ of $1+p\Z_p$ 
(called the wild character of $\kappa$) and an integer $k\geq 2$ (called the weight of $\kappa$). We introduce the metaplectic cover $\tilde{\mathbb{I}}=\mathbb{I}\otimes_\Lambda\tilde{\Lambda}$, where $\tilde{\Lambda}$ is the $\Lambda$-algebra which is equal to 
$\Lambda$ as $\Lambda$-module, but the $\Lambda$-algebra structure is given by 
$(\lambda,x)\mapsto\lambda^2x$ for $\lambda\in\Lambda$ and $x\in\tilde{\Lambda}$. {We say that a point in $\tilde{\mathcal{X}}(\mathbb{I})=\Spec(\tilde{\mathbb{I}})(\C_p)$ is  {arithmetic} if it maps to an arithmetic point of $\mathcal{X}(\mathbb{I})$ via the canonical projection map 
$\tilde{\mathcal{X}}(\mathbb{I})\rightarrow {\mathcal{X}}(\mathbb{I})$. }
A  {$\mathbb{I}$-adic Jacobi form} of index $N$ is a formal power series in 
\[\mathbb{S}=\sum_{r^r\leq 4nN}c_{n,r}q^n\zeta^r\in \tilde{\mathbb{\I}} \pwseries{q,\zeta},\] where $c_{n,r}\in\tilde{\mathbb{I}}$, such that, for each arithmetic point $\kappa\in\tilde{\mathcal{X}}(\mathbb{I})$,  the specialization of $\mathbb{S}(\kappa)$ obtained by evaluating the coefficients $c_{n,r}$ at $\kappa$, is the $(q,\zeta)$-expansion of a classical Jacobi form, with coefficients in $\C$, after fixing an algebraic isomorphism $\C\simeq\C_p$ which we choose from now on. We call  {classical specialisations} those classical Jacobi forms arising as specialisations of a $\mathbb{I}$-adic form.  Restricting to those $\mathbb{I}$-adic forms whose classical specialisations are ordinary cuspidal forms, we define the submodule of {$\mathbb{I}$-adic ordinary cuspidal} forms, 
denoted $\mathbb{J}_N^\mathrm{cusp,ord}(\mathbb{I})$. Let 
$J^{\mathrm{cusp,ord}}_{k,N}(p^s,\chi)$ be the $\C$-vector space of Jacobi cuspidal ordinary forms 
of weight $k$, index $N$, level $p^s$ and character $\chi$; we refer to \S\ref{sec:2.1} for a review of the basic facts on classical Jacobi forms. The main result of the first part is the following theorem, proved under the hypothesis that $\mathbb{I}$ is primitive, meaning that classical specialisations in a dense subset of the weight space are newforms (see Def. \ref{LambdaJacobi} (4)). 

\begin{theorem} Let $\mathbb{I}$ be primitive. The $\Lambda=\mathcal{O}\pwseries{1+p\Z_p}$-module $\mathbb{J}_N^\mathrm{cusp,ord}(\mathbb{I})$ is finitely generated, and for each arithmetic point $\kappa\in\tilde{\mathcal{X}}(\mathbb{I})$ we have 
\[\mathbb{J}_N^\mathrm{cusp,ord}(\mathbb{I})/\mathfrak{p}_\kappa \mathbb{J}_N^\mathrm{ord,cusp}(\mathbb{I})\simeq 
J^{\mathrm{cusp,ord}}_{k,N}(p^s,\chi)\] 
where $\mathfrak{p}_\kappa$ is the kernel of $\kappa$, and 
the integers $k,s$ and the character $\chi$ depend on $\kappa$ only. 
\end{theorem}

In the second part of the paper, 
we set up a $p$-adic theta correspondence relating ordinary $p$-adic families of 
elliptic cuspforms with ordinary $p$-adic families of Jacobi forms. Recall that the classical theta correspondence associates, under certain arithmetic conditions, a Jacobi form to an elliptic modular form; see \S \ref{sec-theta} for a utilitarian recapitulation on the theta correspondence. We develop a $p$-adic version of this theta correspondence. 
The main result of this part is Theorem \ref{MainFamilies}, which shows the existence of a $p$-adic family of Jacobi forms interpolating classical theta lifts of the classical forms in a given $p$-adic family of elliptic ordinary cuspforms. This $p$-adic version of the theta correspondence has been developed for $p$-adic half-integral weight modular forms by Stevens \cite{St} in the elliptic case, as recalled above, and the authors \cite{LN-Doc} in the quaternionic case. 
To state our main result in a more precise form, recall that for an elliptic modular form $f\in S_k(\Gamma_1(M),\chi^2)$ of weight $k$, level $\Gamma_1(M)$ for some integer $M$, and character $\chi^2$,   
we have a Shintani lift \[f\longmapsto \mathcal{S}_{(D_0,r_0)}^{(\chi)}(f)\] to the space of Jacobi forms of prescribed weight, index, character and level depending on $f$ and the choice of a fundamental discriminant $D_0$ which is a square modulo $N$, a square root $r_0$ of $D_0$ modulo $N$, and a square root $\chi$ of the character $\chi^2$; 
the precise definition of the map $\mathcal{S}_{(D_0,r_0)}^{(\chi)}$ is recalled in  \S\ref{sec-theta}. 
Fix a Hida family $f_\infty=\sum_{n\geq 1}a_nq^n\in \mathbb{I}\pwseries{q}$ of ordinary elliptic cusp forms; recall that for each arithmetic point 
$\kappa\in \mathcal{X}(\mathbb{I})$ the specialization $f_\kappa=\sum_{n\geq 1}a_n(\kappa)q^n$ is an elliptic eigenform of level $\Gamma_1(Np^s)$, character $\chi$ and weight $k$, where the integers $k,s$ and the character $\chi$ depend on $\kappa$ only, and $p^s$ is the maximum between $1$ and the $p$-power of the conductor of $\chi$. 

\begin{theorem}\label{thmB} There exists a $p$-adic family of Jacobi forms $\mathbb{S}$ such that for each 
$\tilde\kappa\in\tilde{\mathcal{X}}(\mathbb{I})$ lying over an arithmetic point $\kappa\in {\mathcal{X}}(\mathbb{I})$ of character $\chi^2$ with $\chi$ a non-trivial character modulo $p$, we have 
\[\mathbb{S}(\tilde\kappa)=\lambda(\kappa)\cdot S_{D_0,r_0}^{(\chi)}(f_\kappa)\]
where $\lambda(\kappa)\in\C_p$ is a $p$-adic period which is non-zero on an open subset of $\mathcal{X}(\mathbb{I})$. 
\end{theorem}

A more general form of Theorem \ref{thmB} is presented in Theorem \ref{MainFamilies} and Theorem \ref{MainFamiliesTrCh}
which consider generic arithmetic points $\kappa$ of non-trivial character $\chi^2$ and trivial character, respectively. These results show that theta lifts can be interpolated in families. We also discuss a result, see Theorem \ref{new forms}, where we consider similar questions for forms which arise in Hida families as ordinary $p$-stabilisations of forms of level $N$ prime to $p$. 

\subsection*{Acknowledgments} Part of this work was done during visits of 
M.-H.N. at the Mathematics Department of the University of Padova whose congenial hospitality he is grateful for, and also during a visit of M.L. in Montr\'eal supported by the grant of the CRM-Simons professorship held by M.-H.N. in 2017-2018 at the Centre de recherches math\'ematiques (C.R.M., Montr\'eal).

\section{$\Lambda$-adic families of Jacobi forms} 

In this section, we give an account of some aspects of Hida theory for families of Jacobi forms relying on the well-known close relationship between Jacobi forms and modular forms of half-integral weight. Explicit examples of $\Lambda$-adic families of Jacobi forms are provided by theta lifts of $\Lambda$-adic families of classical modular forms, and we provide them later in the next section. It should be no surprise that Hida theory extends to the setting of Jacobi forms: Wiles's convolution trick with the classical $\Lambda$-adic Eisenstein series readily shows the existence of a $\Lambda$-adic family of Jacobi forms specializing to a given $p$-ordinary Jacobi form, see Section \ref{Wilestrick} for details. 

We rely on two main tools to study $p$-adic families of Jacobi forms: 

\begin{enumerate}
\item

first, Candelori's unpublished results in the preprint \cite{CandeloriTheta} applied to classical modular curves to relate algebraically Jacobi forms and vector-valued half-integral weight modular forms.
While \cite{CandeloriTheta} treats the higher-dimensional case using the language of stacks, all we need can be formulated plainly and with completeness in terms of modular curves without relying on his more sophisticated techniques.

\item second, Hida's treatment for scalar-valued half-integral weight modular forms \cite{HarHalf}.
Recall that Hida developped in \cite{HarHalf} his eponymous theory for {scalar-valued} half-integral weight modular forms, while by Eichler-Zagier's theorem and its generalizations, Jacobi forms are more naturally related to vector-valued half-integral weight modular forms. Since the Jacobi group is not reductive, the geometric approach working well to develop Hida theory for Shimura varieties does not seem to be directly applicable (but see \cite{Kramer95} for groundwork on the arithmetic theory of Jacobi forms), hence our heavy reliance on Hida's paper \cite{HarHalf} for content and overall strategy. Note that for Jacobi forms, the usual role of the tame level is played by the index $m=N$, using the notations below.

\end{enumerate}

\subsection{Jacobi forms}\label{sec:2.1}

We review the basic definition of Jacobi forms. The Jacobi group is the semi-direct product $\SL_2(\Z) \ltimes \Z^2$ where the law is defined by: 
\[(\gamma, \underline{Z}) (\gamma', \underline{Z'}) = (\gamma \gamma', \underline{Z} \gamma' + \underline{Z'}).\] The Jacobi group acts on the $\calH \times \C$. Let $\Gamma^J := \Gamma \ltimes \Z^2$, $\Gamma$ a congruence subgroup of $\SL_2(\Z)$. Let $\underline{Z} = (\lambda,\mu) \in \Z^2$, and $\gamma = \smallmat {a}{b}{c}{d} \in \Gamma.$ For $(\gamma, \underline{Z}) \in \Gamma^J$ and $(\tau, z) \in \calH \times \C$, let 
\[(\gamma, \underline{Z}) \cdot (\tau, z) = \Big( \frac{a\tau +b}{c\tau +d}, \frac{z + \lambda \tau + \mu}{c\tau +d} \Big).\] This defines a right action of $\Gamma^J$ on $\calH \times \C$. Further, the Jacobi group acts on complex functions on $\calH \times \C$ via the slash operator $\phi_{| k,m}$ for integers $k,m \in \N$ i.e., $(\gamma, \underline{0})$ acts as \[(c\tau + d)^{-k} e \Big( \frac{-cmz^2}{c\tau +d} \Big) \phi \Big( \frac{a\tau + b}{c\tau +d}, \frac{z}{c\tau +d} \Big),\] and $(1, (\lambda,\mu))$ acts as \[e(\lambda^2 m \tau + 2 \lambda m z) \phi (\tau, z + \lambda \tau + \mu),\] where $e(-)$ represents the exponential $e^{2\pi i -}$.
\begin{dfn} Let $k,m \in \N$. A Jacobi form of weight $k$ and index $m$ for the congruence subgroup $\Gamma$ is a function:
\[ \phi: \calH \times \C \arr \C \] which satisfies the following condition:
\begin{enumerate}
\item it satisfies $\phi_{| k,m} (\gamma, \underline{Z})= \phi$ for all $(\gamma, \underline{Z}) \in \Gamma^J$.
\item it is holomorphic on $\calH \times \C$;
\item  it is holomorphic at the cusps.
\end{enumerate}

\noindent
Here, we refer to \cite[Def. p.9]{EZ} for a precise explanation of the behaviour at the cusps in terms of the Fourier expansion. Further, a Jacobi cusp form is defined as Jacobi form vanishing at all cusps in a suitable sense.  The $\C$-vector space 
of Jacobi (resp. cusp) forms of weight $k$ and index $m$ for a congruence subgroup $\Gamma$ is denoted by $J_{k,m}(\Gamma)$ (resp. $J^\mathrm{cusp}_{k,m}(\Gamma))$ by slight abuse of notation.

\end{dfn}

\subsection{The Eichler-Zagier theorem}

We describe in this section an account of 
results by \cite{CandeloriTheta} which allow to 
derive geometrically a higher level version of the classical Eichler-Zagier theorem. We state results in \cite{CandeloriTheta} in the $1$-dimensional case of modular curves; the reader is advised to keep a copy of \cite{CandeloriTheta} when reading this subsection to compare our setting with the more general results developed in \emph{loc cit.}, see also \cite{CandeloriThesis} for background and more details.
The strategy, aimed to introduce $p$-adic families of Jacobi forms, is to add level structures of increasing $p$-power levels to obtain an analogue of  \cite[Thm.4.2.1]{CandeloriTheta}. After introducing analogous level structures on the $\C$-vector space of vector-valued modular forms, and passing to global sections over $\C$, we recover a generalization of the Eichler-Zagier theorem in higher level (see also \cite{Boylan}). 
A more general version of these results could also be obtained by introducing more general level structure; however, adding $p$-power level structure is enough for our applications to $p$-adic families of Jacobi forms. As mentioned above
we provide some of the details to accomplish this algebraically in the technically much simpler 1-dimensional case, see \cite[Exa. 2.6.10, Sect. 2.8]{CandeloriTheta} for the simplifications occuring, especially the ``accidental isomorphism'' of \cite[Eq. (2.8)]{CandeloriTheta} which does not extend to higher degree. We recall a few definitions from \cite{CandeloriTheta} but we specialize them readily to modular curves, cf. \cite{Ramsey2006}. 
We shall formulate directly the results in term of compact modular curves $X_0(4N)$ and $X_1(p^s)$, obtained as compactification of the open modular curves $Y_0(4N)$ and $Y_1(p^s)$ associated to the congruence subgroups $\Gamma_0(4N)$ and $\Gamma_1(p^s)$, for an integer $s\geq 1$. 
Note that because of its emphasis on the higher dimensional set-up, \cite{CandeloriTheta} does not treat e.g., compactifications and Hecke operators, but \cite{HarHalf} and \cite{Ramsey2006} do, see \cite{DeligneRapoport}, cf. \cite[1.2.4]{MBCompositio} for details on stacky compactifications of modular curves. Relying on algebraic geometry as in \cite{CandeloriTheta}, \cite{Ramsey2006} and \cite{Kramer91} suggests the wider generality available, while recovering optimal, canonical results (also available through more elementary, classical techniques alone) once specializing to $\C$. 

Let $N \in \N$ be an integer and $p\nmid N$, $p\geq 5$ a prime number. 
The stack $\calA_{1,1}(\Theta)$, introduced in \cite[Def. 2.6.6]{CandeloriTheta}, is (isomorphic to) the classical open moduli stack $Y_0(4N)$ of pairs $(E,H)$ of elliptic curves $E$ equipped with a subgroup scheme $H$ of $E[4N]$ locally isomorphic to $\Z/4N\Z$.  
The theta group is defined as \[ \Gamma(1,2) := \left\{ \smallmat{a}{b}{c}{d} \in \SL_2(\Z) | ab,cd \equiv 0 \mod{2} \right\}.\]
Following \cite[Def. 2.6.1]{CandeloriTheta}, we define the stack $\calX_{1,1}$ as classifying elliptic curves together with a symmetric, relatively ample, normalized invertible sheaf $\calL$ of degree $1$. 
The stack $\calX_{1,1}$ maps to the moduli stack of elliptic curves, and it is smooth over $\Z[1/2]$.

According to a result of Moret-Bailly (see \cite[Rmk. 3.4.2]{CandeloriTheta}), the stack $\calX_{1,1}$ has two connected components: we therefore have a decomposition $\calX_{1,1} = \calX^+_{1,1} \sqcup \calX^-_{1,1}$. More precisely, as it is pointed out in \cite[p. 16]{CandeloriTheta}, over $\C$ we have that: \[ \calX^{+,\mathrm{an}}_{1,1} = \Gamma(1,2) \backslash \calH, \quad \calX^{-,\mathrm{an}}_{1,1} = \SL_2(\Z) \backslash \calH. \]
Here the superscript $^\mathrm{an}$ denotes the associated complex analytic space. We denote by the more customary notation $Y(\Gamma(1,2)) :=  \calX^+_{1,1} $ the so-called ``even characteristic" connected component of the stack defined over $\Z[1/2]$.

There is a natural map $\Des: Y_0(4N) \arr \calX_{1,1}$ defined in greater generality using the moduli interpretations, see \cite[Defining equation (2.9)]{CandeloriTheta}, and it factors through $\calX^{+}_{1,1}$ by \cite[Sec. 2.7]{CandeloriTheta}. Over modular curves, we can define it in the following concrete way:
\begin{dfn} (\cite[(2.9) and Lemma 2.8.6]{CandeloriTheta})  We denote by 
\[ \Des_{\C}: Y_0(4N)_{\C} \arr Y(\Gamma(1,2))_{\C} , \]
\noindent
the analytification of the descent map induced by the homomorphism between the fundamental groups:   \[ \Gamma_0(4N) \arr \Gamma(1,2), \]
 \[ \Des \mat{a}{b}{c}{d} = \mat{a}{2bN}{c/2N}{d}. \]

\end{dfn}

We now introduce the automorphic sheaves giving rise to Jacobi forms, resp. half-integral weight modular forms. 

On the Jacobi side, the basic invertible sheaf is $\calL_{2N}$, and for simplicity we describe it using complex algebraic geometry i.e., using the complex uniformization (see \cite{Kramer91} for the treatment of general level structures and \cite{Kramer95} for an approach relying on arithmetic geometry including the treatment of compactifications). Consider the elliptic modular surface given by the quotient 
\[\calS_{\Gamma_1(p^s)} := \calH \times \C / (\Gamma_1(p^s) \ltimes \Z^2),\] where $(m,n) \in \Z^2$ acts by $(\tau, z) \mapsto (\tau, z + m + n\tau)$, and $ \gamma $ acts by \[ \gamma(\tau, z) = \Big( \frac{a\tau + b}{c\tau +d}, \frac{z}{c\tau+d} \Big), \gamma = \mat{a}{b}{c}{d} \in \Gamma_1(p^s).\]
Define $\calL_{2N} := \calO_{\calS} (2N e_{\calS}) \otimes \Big( \Omega^1_{\calS/Y_1(p^s)} \Big)^{2N}, $ where $e_{\calS}$ is the identity section. Let $\pi: \calS \rightarrow Y_1(p^s)$ denote the projection morphism, and define $\calJ_{1,2N} :=  \pi_* \calL_{2N}$. Let $\calE \rightarrow Y_1(p^s)$ denote the universal elliptic curve over $Y_1(p^s)$. 

\begin{definition} Let $N\geq 1$ and $k\geq1$ be integers. A  {Jacobi form of index $N$ and weight $k$ with respect to $\Gamma_1(p^s)$} is a global section of $\calJ_{1,2N} \otimes  \omega^{\otimes k}_{\calE/Y_1(p^s)} $. \end{definition}

On half-integral weight side, the construction involves a few more ingredients. To tackle half-integral weight modular forms geometrically, the standard trick is to fix a square root of the Hodge bundle $\omega$. This is the strategy exploited by Hida in \cite{HarHalf}. Candelori's so-called metaplectic stacky formalism is a generalisation of this idea.
The canonical bundle over $\calX_{1,1}$ of weight 1/2 forms denoted $\underline{\omega}_{\Theta}^{1/2}$ is very closely connected to the square root of the Hodge bundle $\omega$ (see \cite[Cor. 3.3.3]{CandeloriTheta}), hence the notation. Its definition is as follows:
the invertible sheaf $\underline{\omega}_{\Theta}^{1/2}$ is defined on \cite[p.22, line 3]{CandeloriTheta} as the canonical square root of $\calM_{\Theta} \otimes \omega$ induced by the canonical isomorphism of \cite[Thm. 5.0.1]{CandeloriAlg}, where $\calM_{\Theta}$ is the theta multiplier bundle of \cite[Def. 3.2.2]{CandeloriTheta}. That is, the choice of the square root of $\omega$ is made compatibly with the choice of the square root of $\calM_{\Theta}$ so that their tensor product is canonical.

More concretely, the bundle $\Des^* \underline{\omega}_{\Theta}^{1/2}$ can be seen as an invertible sheaf of modular forms of half-integral weight $1/2$ on $Y_0(4N)$, prompting the following:

\begin{definition} (\cite[Def. 3.4.1]{CandeloriTheta}) Let $k\geq 1$ be an odd integer. A  {scalar-valued modular form of weight $k/2$} is a global section of $\Des^* \underline{\omega}_{\Theta}^{k/2}$. 
\end{definition}

We now introduce vector-valued modular forms on the modular curve $Y_1(p^s)$. The new ingredient is the so-called Weil bundle. Over $\C$, this is the local system corresponding to the Weil representation, and for modular curves, it is defined as follows. Let $H(2N) := \Z/2N\Z$, and $K(2N) := \Z/2N\Z \times \Z/2N\Z$. The Heisenberg group of type $2N$ is defined as 
\[G(2N) := \G_m \times K(2N)\] with group structure given by $(\lambda_1, x_1, y_1) (\lambda_2, x_2, y_2) = (\lambda_1 \lambda_2 \langle x_1, y_2 \rangle_{2N}, x_1 +x_2, y_1 + y_2 )$, where $\langle -,- \rangle$ is the standard symplectic pairing of type $2N$ defined in \cite[Eq. (2.4)]{CandeloriTheta}. 
Let $S$ be any scheme. The sheaf $\calV(2N)$ is the free $\calO_S$-module of functions $f: H(2N) \rightarrow \calO_S$, with canonical basis given by delta functions. The Schr\"odinger representation of $G(2N)$ is defined in \cite[Def. 4.1.2]{CandeloriTheta} as the module $\calV(2N)$ given by functions $f: H(2N) \rightarrow \calO_S$, with $G(2N)$-action given by:
\[ \rho(\lambda, x,y) f(y') := \lambda \langle x,y' \rangle f(y'+y). \]
The sheaf $\calV(2N)$ is the Schr\"odinger representation of weight one and rank $2N$. 
We can now define the Weil bundle $\calW_{1,2N}$ as the locally free sheaf $\calV(2N)^{\vee} \otimes \Des^* \calM_{\Theta}^{-1/2} \otimes \omega^{-1/2}, $ see \cite[Eq. (4.3)]{CandeloriTheta}. The key point of \cite[Sec. 4.3]{CandeloriTheta} is that the Weil bundle is defined over the metaplectic stack (even though neither $\calV(2N)^{\vee}$ nor $\Des^* \calM_{\Theta}^{-1/2}$ is), i.e., it is defined over $Y_1(p^s)$ after fixing a square root of the Hodge bundle $\omega$ over $Y_1(p^s)$. 

\begin{definition} Let $k\geq 1$ be an odd integer. A  {$\calW_{1,2N}$-vector valued modular form of weight $k/2$}
is a global section of the vector bundle $\calW_{1,2N} \otimes \omega^{k+1/2}$. \end{definition}

\begin{remark} The above definition of vector-valued modular forms of weight $k/2$ corresponds to that of \cite[Def. 4.3.3]{CandeloriTheta}. \end{remark}
 
Summing up, we get the following comparison isomorphism cf. \cite[Thm. 4.2.1, Equ. (4.3)]{CandeloriTheta}:

\begin{thm} 
There is a canonical isomorphism
\[ \calJ_{1, 2N} \overset{\cong}{\arr}  \calW_{1,2N} ,\]

\noindent 
as locally free sheaves of rank $2N$ over $Y_1(p^s)_{\C}$.
\end{thm}

\begin{proof}
The key point in the above construction is that both sheaves are generally defined over the metaplectic stack over the modular elliptic curve $Y_1(p^s)$, that is, over $Y_1(p^s)$ itself once a square root of the Hodge bundle $\omega_{\calE/Y_1(p^s)}$ is fixed. The isomorphism \cite[(4.3)]{CandeloriTheta} states that
\[ \calV(2n) \otimes \Des^* \calM_{\Theta}^{-1/2} \cong \calJ_{1,2N} \otimes \omega^{1/2},\]
\noindent
and it is canonical thanks to \cite[Thm 5.0.1]{CandeloriAlg}.
\end{proof}

\begin{rmk}
\cite[Thm. 4.2.1, (4.3)]{CandeloriTheta} refine an old result of Mumford, see \cite[Sect. 4.4]{CandeloriTheta} and \cite{MBCompositio} for details. We work over $\C$ as the most precise result is established up to $\pm 1$ only in that setting in \cite[Thm 4.4.2]{CandeloriTheta}, see \cite[Rmk 5.0.3]{CandeloriAlg} for a description of the ambiguity in general.
\end{rmk}

After  tensoring with powers of the Hodge bundle and extending to the compactification $X_1(p^s)$ (cf. \cite[Sect. 2.5]{MBCompositio} for the computation at the cusps with Tate curves), we obtain the following slight extension to level $\Gamma_1(p^s)$ of the classical Eichler-Zagier isomorphism, see \cite[Thm.5.1]{EZ}, cf. \cite[Cor.4.5.1]{CandeloriTheta} and see \cite[\S 3, Thm. 3.5]{Boylan} for further generalization over $\C$:

\begin{cor} \label{EZgen}
There is a canonical isomorphism between the space $J_{k,N}(\Gamma_1(p^s))$ of Jacobi forms of index $N$, weight $k$ and level $\Gamma_1(p^s)$ and the space of $\calW_{1,2N}^{\vee}$-vector-valued modular forms of weight $k - 1/2$, denoted $\gerV_{k-1/2, 2N}(\Gamma_1(p^s))$ .
\end{cor}

\begin{proof}
We tensor the isomorphism \[ \calV(2N) \otimes \Des^* \calM_{\Theta}^{-1/2} \cong \calJ_{1,2N} \otimes \omega^{1/2}\] 

\noindent with $\omega^{\otimes k}_{{\calE}/Y_1(p^s)}$ on both sides and obtain the canonical isomorphism:
\[ \calV(2N) \otimes \Des^* \calM_{\Theta}^{-1/2} \otimes \omega^{\otimes k-1/2}_{{\calE}/Y_1(p^s)} \cong \calJ_{1,2N} \otimes \omega^{\otimes k}_{{\calE}/Y_1(p^s)}.\]
\noindent
The sheaves and the isomorphism extend to the compactification $X_1(p^s)$. 
\noindent
Taking global sections yields the canonical isomorphism between the two spaces of modular forms.
\end{proof}

\noindent
The vector-valued modular forms in $\gerV_{k-1/2, 2N}(\Gamma_1(p^s))$ above generalize the $2N$-tuples of modular forms as in \cite[p. 59]{EZ} with respect to $\SL_2(\Z)$.

\subsection{Density of classical Jacobi forms}

As the theory of Jacobi forms is not fully geometrized in comparison with the theory of classical modular forms, we face similar difficulties as for half-integral weight modular forms where the underlying non-reductive group is the metaplectic group. Our strategy to prove the density theorem for classical Jacobi forms within $p$-adic Jacobi forms is to use the natural isomorphism \`a la Eichler-Zagier between the space of Jacobi forms and the space of vector-valued half-integral weight modular forms. More precisely, we show that the completion of the space of $p$-power level classical Jacobi forms does not depend on the weight, and therefore any such classical subspace is dense in the space of $p$-adic Jacobi forms, see Thm \ref{densitythm} for the precise statement.

We set up the necessary notation for the density theorem. Let\[J_{k,N}(A) := \bigcup_s J_{k,N}(\Gamma_1(p^s);A),\qquad 
J^\mathrm{cusp}_{k,N}(A) := \bigcup_s J^\mathrm{cusp}_{k,N}(\Gamma_1(p^s);A),\] for $A$ any subalgebra of $\C$, and where $J_{k,N}(\Gamma_1(p^s);A)$ denotes the space of Jacobi forms of weight $k$, index $N$ and level $\Gamma_1(p^s)$ with Fourier-Jacobi coefficients in $A$.

Let $W(\Fp)$ be the ring of Witt vectors with residue field $\Fp$, and $K$ the quotient field of $W(\Fp)$.
We fix once and for all a complex (resp. $p$-adic) embedding of $\bar{\Q}$ into 
$\C$ (resp. $\bar\Q_p$). We may thus take $A  := W(\Fp) \cap \overline{\Q}$ and $J_{k,N}( W(\Fp)) = J_{k,N}(A) \otimes_{A} W(\Fp)$,   $J^\mathrm{cusp}_{k,N}( W(\Fp)) = J^\mathrm{cusp}_{k,N}(A) \otimes_{A} W(\Fp)$. We write $\widehat{J}_{k,N}(W(\Fp))$ for the $p$-adic completion of $J_{k,N}(W(\Fp))$, and similarly for cusp forms. The following theorem proves that the $p$-adic completion is independent of the weight $k$.

\begin{thm}[Density theorem] \label{densitythm} For $k \geq 2$, we have an isomorphism preserving the Fourier-Jacobi coefficients:
\[ \widehat{J}^\mathrm{cusp}_{k,N}(W(\Fp))  \cong \widehat{J}^\mathrm{cusp}_{k+1,N}(W(\Fp)). \]
\end{thm}

\begin{proof}
The first step is to use Corollary \ref{EZgen} to transfer the problem to vector-valued modular forms. Note that the construction of vector-valued modular forms or taking the $p$-adic completion are operations that commute with each other, since the corresponding vector spaces have finite dimension. Using \cite[Thm.1, p.145]{HarHalf}, we already know the result for scalar-valued half-integral weight modular forms, and this implies that the space of classical vector-valued modular forms is dense in the space of $p$-adic vector-valued modular forms by using the projections to the scalar-valued components. The result for Jacobi forms follows.
\end{proof}

\subsection{Ordinary forms} 
Let $N\geq 1$ be an integer, $p\nmid N$, $p> 5$ a prime number, $s\geq 1$ an integer 
and $ \chi:(\Z/p^s\Z)^\times\rightarrow\bar\Q_p^\times $ a Dirichlet character. 
Denote $\mathrm{U}_J(p)$ the Hecke operator at $p$ acting on the space of Jacobi forms $J^\mathrm{cusp}_{k,N}(p^s,\chi)$ of index $N$, character $\chi$, with respect to $\Gamma_0(p^s)$, see \cite{MRVSK} and \cite{Ibu} for details. By \cite[\S3, Eq. (5), p.165]{MRVSK}, if the $(n,r)$-th Fourier-Jacobi coefficient of 
a $\mathrm{U}_J(p)$-eigenform $F$ is $c(n,r)$, then the $(n,r)$-th Fourier-Jacobi coefficient of $F|\U_J(p)$ is $c(p^2n, p r)$.

\begin{dfn}
We say that a Jacobi eigenform $F\in J^\mathrm{cusp}_{k,N}(p^s,\chi)$ is  {$p$-ordinary} if its $\U_J(p)$-eigenvalue is a $p$-adic unit.
\end{dfn}

We denote $J^\mathrm{cusp,ord}_{k,N}(p^s,\chi)$ the $\C$-vector space of ordinary Jacobi cuspforms. 
We start with recalling the notion of newforms in the Jacobi case introduced in \cite{MR}.  
Denote $S_{2k-2}^\mathrm{new}(\Gamma_0(Np^s),\chi^2)$ the space of newforms of 
weight $2k-2$, level $\Gamma_0(Np^s)$ and character $\chi^2$. As in \cite[Sect. 5.1]{MR}, let $f_1,\dots,f_t$ be an orthonormal basis of $S_{2k-2}^\mathrm{new}(\Gamma_0(Np^s),\chi^2)$ 
of normalized Hecke eigenforms, and write $f_i=\sum_{n\geq 1}a_n(f_i)q^n$.  
Let $J^\mathrm{cusp,\new}_{k,N}(p^s,\chi,f_i)$ be the subspace 
of $J^\mathrm{cusp}_{k,N}(p^s,\chi)$ consisting of forms $F$ such that 
$F|\mathrm{T}_J(\ell) = a_\ell(f_i)\cdot F$ for all primes $\ell\nmid Np$. Define 
\[J^\mathrm{cusp,\new}_{k,N}(p^s,\chi)=\bigoplus_{i=1}^tJ^\mathrm{cusp,\new}_{k,N}(p^s,\chi,f_i).\] 

\begin{thm}[Weak control theorem for Jacobi newforms] \label{weakcontrol} Let $k \geq 2$.
The number of linearly independent $p$-ordinary Jacobi cuspidal newforms in $J_{k,N}^{\mathrm{cusp},\new}(p^s,\chi)$ is bounded above independently of the weight $k$. 
\end{thm}

\begin{proof}
Using \cite[Thm. 5.2]{MR}, we embed the space of Jacobi newforms which are $p$-ordinary injectively into the space of half-integral weight newforms via the generalized Eichler-Zagier map denoted by $\calZ_N$ (or rather $\calZ_m$) in \cite{MR}. Since the map $\calZ_N$ preserves $p$-ordinarity, the image is in the space of $p$-ordinary half-integral weight newforms. To conclude, we apply \cite[Prop.3]{HarHalf} which relies heavily on results of Waldspurger via the Shimura correspondence.
\end{proof}

\subsection{Hida families for $\GL_2$} \label{sec:HidaFamilies} 

We set up the notation for Hida families of modular forms. 
Let $N\geq 1$ be an odd positive integer, and let $p> 5$ be a prime number such that $(p,N)=1$.
Let \[f_\infty(\kappa)=\sum_{m\geq 1}a_n(\kappa)q^n\in\mathcal R\pwseries{q}\] be a primitive Hida family, 
where $\mathcal R$ is a primitive branch of the ordinary Big Hida Hecke algebra acting of forms of tame level 
$\Gamma_1(N)$ with coefficients in the ring of integers $\mathcal O$ of a finite extension $F$ of $\Q_p$; for details and definitions, see 
\cite[\S 2.1]{Ho2}, whose notations we are following, 
as well as \cite{Hida85}, \cite{Hida86}, \cite{Hida88}, \cite[\S 12.7]{Nek} and \cite{LV2}.
We just recall that $\mathcal R$ is an integral complete local noetherian domain, which is finitely generated 
over the Iwasawa algebra $\Lambda=\mathcal O[\![1+p\Z_p]\!]$, and such that for 
each arithmetic weight $\kappa$ in 
\[\mathcal X^\mathrm{arith}(\mathcal R)\subseteq\mathcal X(\mathcal R)=\Hom_{\mathcal O\text{-alg}}^\mathrm{cont}(\mathcal R,\bar\Q_p)\] 
of $\mathcal R$, the formal 
Fourier expansion $f_\kappa=\sum_{n\geq 1}a_n(\kappa)q^n$, where $a_n(\kappa):=\kappa(a_n)$, is an ordinary $p$-stabilized newform of level $N$. 
As a general notation, for $a\in\mathcal{R}$ and 
$\kappa\in\mathcal{X}(\mathcal{R})$ we write $a(\kappa)$ for $\kappa(a)$. 
We also recall that a homomorphism 
$\kappa\in\mathcal X(\mathcal R)$ is said to be 
 {arithmetic} if its restriction to $1+p\Z_p$ has the form 
$\gamma\mapsto \psi(\gamma)\gamma^{k-2}$
for a finite order character $\psi$ (called the  {wild character} of $\kappa$) 
and an integer $k\geq 2$ (called the  {weight} of $\kappa$).   
In this case the weight of $f_\kappa$ is $k$, its level is $\Gamma_s=\Gamma_1(N)\cap\Gamma_1(p^s)$ 
where $s\geq 1$ is the maximum between the conductor of $\psi$ and $1$, 
and the coefficients $a_n(\kappa)$ belong to the finite 
extension $\calR_\kappa:=\mathcal R_{\p_\kappa}/\p_\kappa\mathcal R_{\p_\kappa}$ of $\Q_p$, 
where $\p_\kappa:=\ker(\kappa)$ and $\mathcal R_{\p_\kappa}$ is the localization of $\mathcal R$ at $\kappa$.
We also introduce the notation $\mathcal O_{f_\kappa}$ for the valuation ring of $\calR_{\kappa}$. 
We call $(\chi,k)$ the  {signature} of the arithmetic character $\kappa$, see \cite[Def. 1.2.2]{St}.

\subsection{Hida families of Jacobi forms}\label{Jacobi forms subsec}

We recall the notation used in \cite{HarHalf}. Let as above $\Lambda$ be the complete group algebra of $1+p\Z_p$ with coefficients in $\mathcal{O}$, which is non-canonically isomorphic to the one-variable power series ring in a variable $X$ with coefficients in $\mathcal{O}$ via the map $u \mapsto 1 + X$, after fixing a generator $u \in 1+p\Z_p$. Fix  an algebraic closure $\overline{\L}$ of the quotient field $\L$ of $\Lambda$. For each normal integral domain $\I$ in $\overline{\L}$ finite over $\Lambda$ (for example, one can take $\I=\mathcal{R}$ with $\mathcal{R}$ as in \S \ref{sec:HidaFamilies}), let $\calX(\I)$ be weight space i.e., the space of $\C_p$-valued points of $\Spec(\I)$. We say that a weight $\kappa\in\mathcal{X}(\I) $ is arithmetic if its restriction to $\Lambda$ is, in the sense of the previous section. 

\begin{definition} 
The {metaplectic cover} of $\I$ is the $\I$-algebra $\tilde{\mathbb{I}}=\mathbb{I}\otimes_\Lambda\tilde{\Lambda}$, where $\tilde{\Lambda}=\Lambda$ as $\Lambda$-module, but the $\Lambda$-algebra structure is given by 
$(\lambda,x)\mapsto\lambda^2x$ for $\lambda\in\Lambda$ and $x\in\tilde{\Lambda}$. \end{definition}
Put $\tilde{\mathcal{X}}(\mathbb{I})=\Spec(\tilde{\mathbb{I}})(\C_p)$. 
We have a canonical surjective map $\pi:\tilde{\mathcal{X}}(\mathbb{I})\twoheadrightarrow {\mathcal{X}}(\mathbb{I})$, which makes $\tilde{\mathcal{X}}(\mathbb{I})$ a $2$-fold cover of 
${\mathcal{X}}(\mathbb{I})$. 

\begin{definition} 
We say that a point $\tilde\kappa\in\tilde{\mathcal{X}}(\mathbb{I})$ is   {arithmetic} if 
$\pi(\tilde\kappa)$ is arithmetic.\end{definition}

\begin{remark} Our examples arising from theta liftings in the next section are naturally expressed in terms of the metaplectic cover $\widetilde{\mathcal X}(\I)$ obtained from twisting the $\Lambda$-algebra structure by $\sigma(t)= t^2$ on $1+p\Z_p$, because of the choices related to lifting data, see \cite[\S 3 and 6]{St} or \cite[\S 3.3]{LN-Doc}.\end{remark}

\begin{dfn} \label{LambdaJacobi} 
\begin{enumerate}
\item A  {$\Lambda$-adic Jacobi form of index $N$ over $\I$}, or simply a  {$\I$-adic modular form}, is a formal power series in $\tilde{\I}[\![q,\zeta]\!]$: 
\[F_\infty (\tilde{\kappa})= \sum_{r^2 \leq 4nN} c_{n,r}(\tilde{\kappa}) q^n \zeta^r,  \]
such that the specializations $F_{\tilde{\kappa}} :=F_\infty(\tilde{\kappa})$ 
at almost all arithmetic primes $\kappa\in\tilde{\mathcal{X}}(\I)$ are images under the fixed embedding $\overline{\Q} \hookrightarrow \overline{\Q}_p$ of Fourier expansions of classical Jacobi forms of index $N$, weight $k$, level $p^s$, for some integer $s$, and character $\chi_{\kappa}$ of conductor dividing $Np^s$, which are eigenforms for all Hecke operators. We call  {classical specialisations} those classical Jacobi forms which arise as specialisations at arithmetic primes. We finally write $\J_{N}(\I)$ for the $\I$-module of $\I$-adic modular forms. 
\item The $\I$-module of  {$p$-ordinary $\I$-adic Jacobi forms} 
is the $\I$-submodule of $\J_{N}(\I)$ consisting of $\I$-adic forms whose all 
classical specialisations are ordinary. 
\item The $\I$-module $\J_{N}^{\ord, \cusp}(\I)$ of  {ordinary $\I$-adic Jacobi cusp forms} is defined as the $\I$-submodule of 
the $\I$-module $\J_{N}^{\ord}(\I)$ consisting of formal power series whose classical specialisations are cuspforms.
\item We say that $\I$ is  {primitive} if for any $F\in\J_{N}^{\ord, \cusp}(\I)$
there is a Zariski open set $U\in\tilde{\mathcal{X}}(\I)$ such that $F(\kappa)$ is a newform for every $\kappa\in U$. 
\end{enumerate}
\end{dfn}

Similarly, the module of cusp forms is defined by adding the condition of being cuspidal. We denote $\J_{N}^{\ord, \cusp}(\I)$ the $\I$-module of ordinary cuspidal forms. 

\begin{rmk} Examples of primitive $\Lambda$-algebras $\I$ arise naturally when one considers, as we will do in the next section, 
Theta lifts of primitive branches of Hida families.\end{rmk}

\begin{thm}\label{thm 2.13} Let $\I$ be primitive. 
The module of $\Lambda$-adic ordinary cusp forms $\J_{N}^{\ord, \cusp}(\I)$ is free of finite rank over the Iwasawa algebra $\Lambda$.\end{thm}

\begin{proof}
As in \cite{HarHalf}, we follow the argument of Wiles. The main observation is that we can pick an arithmetic weight such that all the specializations of a maximal finite set of linearly independent elements in the space of $p$-ordinary forms are newforms, and then apply Theorem \ref{weakcontrol}. We give the details by paraphrasing Hida's account in \cite[Prop.4, p. 159]{HarHalf}. Let $\J_{N}^{\cusp, \ord}(\I) \otimes \K$ be the $\K$-vector space of $p$-ordinary Jacobi modular forms of index $N$, where $\K$ is the fraction field of $\I$. Let $F_1, \dots, F_m$ be a finite set of linearly independent elements in $\J_{N}^{\cusp, \ord}(\I)$ over $\I$, with Fourier coefficients denoted by $c_{F_j}(n,r)$. Then we can find pairs of integers $(n_i, r_i)$, $i = 1, \dots, m$ so that
\[ D = \det(c_{F_j}(n_i,r_i)) \neq 0, \quad 1 \leq i,j \leq m.\]

\noindent
By the density theorem and the fact that $\I$ is primitive we may choose an arithmetic weight $\kappa$ so that for all $i=1,\dots, m$,
\[ F_i(\kappa) \in J_{\kappa, N}^{\new, \ord} (p^{s(\kappa)}, \chi (\kappa) ) \mbox{ and } D(\kappa) \neq 0. \]
\noindent
Since $D(\kappa) = \det(c_{F_j}(n_i,r_i)(\kappa)) \neq 0$, the specializations $F_i(\kappa)$ are also linearly independent. That is,
\[ m \leq \dim J_{\kappa,N}^{\new,\ord}(p^{s(\kappa)}, \chi(\kappa)), \] which is bounded independently of $\kappa$ by Theorem \ref{weakcontrol}. Thus, there is a maximal set of linearly independent elements in $\J_N^{\ord,\cusp}(\I)$. The rest of the proof is identical as in \cite{HarHalf}.
\end{proof}

\begin{cor}\label{coro 2.14} Suppose $\I$ is primitive. 
For all arithmetic points in $\calX(\I)$ such that $\kappa$ is sufficiently large, we have:
\[\J^{\mathrm{cusp},\ord}_N(\I) / \mathfrak{p}_\kappa  \J^{\mathrm{cusp},\ord}_N(\I) \cong 
 J_{k(s),N}^{\cusp,\ord}(p^{s(\kappa)}, \chi) . \]
\end{cor}

\begin{proof}
The surjectivity is guaranteed by Wiles's trick whose details we recall in the next Subsection, and the rest of the argument follows as in \cite[p.161]{HarHalf}.
\end{proof}

\begin{remark}
It might be possible to obtain a more general version of the above theorem and corollary in which the primitivity assumption is suppressed. This could be accomplished by means of \cite[Theorem 5.16]{MR}, which describes, under certain conditions, the decomposition of Jacobi forms into a direct sum of newforms and oldforms, and using the analogue of Theorem \ref{weakcontrol}  for all $\C$-vector spaces of Jacobi forms appearing in this decomposition. However, to obtain an analogue of Theorem \ref{weakcontrol} one needs to understand the kernel of the Eichler-Zagier map, which is not injective on oldforms, and bound it effectively, at least under some arithmetic conditions. In our applications in this paper and in \cite{LN-GKZ}, we only use primitive branches. \end{remark}

\subsection{Wiles's trick for Jacobi forms}
\label{Wilestrick}

Wiles's trick consist in convoluting a classical modular form $f$ with the $\Lambda$-adic Eisenstein series to show the existence of a Hida family specializing to $f$. Viewing the Eisenstein family as indexing a family of Jacobi forms of index $0$, we obtain by convoluting with a Jacobi form of index $m$ a family of Jacobi forms also of index $m$. A clear description of the Wiles trick for classical modular forms can be found in Hida's book \cite[\S 7.1]{HidaEisen} to which we send the reader for more details. It is enough to consider $\I = \Lambda$, thanks to base change property: $ \J_N^{\ord}(\I) \cong \J_N^{\ord}(\Lambda) \otimes_{\Lambda} \I $ as $\I$ is $\Lambda$-free, cf. \cite[argument on last line of p.149 and two first lines of p.150]{HarHalf}.

Denote by $E(\psi)$ the $\Lambda$-adic Eisenstein series defined in \cite[p. 198]{HidaEisen} , where $\psi = \omega^a$ is an even Dirichlet character with $0 \leq a < p-1$, and $\omega$ the Teichm\"uller character.

Consider a Jacobi form $F \in J_{k',m}(\Gamma_0(p^s), \chi)$, and the product $F \cdot E(\psi)$ inside $\Lambda[[q,\zeta]]$. In particular, we get the product of $F$ and the classical Eisenstein series specialized at weight $k$ as a classical Jacobi form:
\begin{equation} \label{equationweightzero} F \cdot E(\psi)(u^k-1) \in J_{k+k',m}(\Gamma_0(p^s), \chi \psi \omega^{-k}). 
\end{equation}

Let us write $F \cdot E(\psi)(X) = \sum_{r^2 \leq 4nm} c_{n,r}(X) q^n \zeta^r$.
Define the convolution product of $F$ and $E(\psi)$ giving a family over $\Lambda$ (cf. \cite[p.200]{HarHalf}):
\[ F * E(\psi)(X) := \sum_{r^2 \leq  4nm} c_{n,r}\Big( u^{-k'}\chi^{-1}(u)X + (u^{-k'} \chi(u)^{-1} - 1 ) \Big) q^n \zeta^r, \]
\noindent that is, we have the specialization $F * E(\psi)(\chi^{-1}(u) u^{k-k'}-1) \in J_{k,m}(\Gamma_0(p^s), \chi \psi \omega^{-k})$ for all $k > k'$ as well as for $k=0$ from Equation (\ref{equationweightzero}) just above, and hence the desired $\Lambda$-adic family of Jacobi forms in the sense of Definition \ref{LambdaJacobi}.

\begin{rmk}
Thanks to the work of Guerzhoy \cite{Gu2}, \cite{Gu3}, the above computation can be generalized using the more involved Jacobi-Eisenstein series of non-trivial index but this gives rise to a non-trivial shift in the index of the resulting family.
\end{rmk}

\section{$\Lambda$-adic theta liftings} 

The results of this section are similar to that obtained by Stevens and Hida in the context of 
the theta lifting to spaces of 
half-integral weight modular forms \cite{St}, \cite{HarHalf}. 
Similar techniques have already been employed in the context of Jacobi forms by 
Guerzoy in \cite{Gu2}, \cite{Gu3}, \cite{GuerzhoyCrelle} and  \cite{Gu1}.

\subsection{Theta liftings} \label{sec-theta}
In this section, we recall the connection between classical Jacobi forms and modular forms via theta liftings, following \cite{EZ} and \cite{MR}.  

We first set up some general notation, which will be used throughout the paper.  
For any integer $M\geq 1$, any Dirichlet character $\chi:(\Z/M\Z)^\times\rightarrow\C^\times$ 
and any $k\in\Z$,  
denote by $S_{k}(\Gamma_0(M),\chi)$ the complex vector space of 
elliptic cuspforms of weight $k$, level $\Gamma_0(M)$
and character $\chi$. If $M={S}\cdot {T}$ and $\chi$ is a character modulo ${T}$, 
we denote by $J^\mathrm{cusp}_{k,{S}}({T},\chi)$ the complex vector space of 
Jacobi forms of weight $k$, level $\Gamma_0({T})$, character $\chi$ and 
index $M$ (see \cite[Ch. I, \S 1]{EZ} for the definition). If $T=1$ (so the character $\chi$ is trivial), 
we denote the corresponding spaces simply by $S_k(\Gamma_0(M))$ 
and $J^\mathrm{cusp}_{k,M}$, respectively.
For $f\in S_{k}(\Gamma_0(M),\chi)$ and $\phi\in J^\mathrm{cusp}_{k,{S}}({T},\chi)$ 
we have the respective Fourier expansion   
$f(z)=\sum_{n\geq 1}a_nq^n$ and Fourier-Jacobi expansion
$\phi(\tau,z) = \sum_{r,n} c(n,r) q^n \zeta^r$, 
where $q = e^{2 \pi i \tau}$ and $\zeta = e^{2 \pi i z}$ 
and the Fourier-Jacobi expansion is taken over all pairs $(n,r)$ of integers such that 
$r^2-4{S}n<0$. 

As a general notation, let $\mathrm{T}(m)$ and $\mathrm{T}_J(m)$  be the Hecke operators at $m \in \N$ 
acting on the space of elliptic modular forms $S_k(\Gamma_0(M),\chi)$ and Jacobi forms 
$J^\mathrm{cusp}_{k,S}(T,\chi)$, respectively. 
If $\ell$ is a prime number, which is not coprime to the level (in the 
case of elliptic modular forms) respectively, not coprime to the index or the level (in the case of Jacobi forms), 
we denote the operators by $\mathrm{U}(\ell)$ and $\mathrm{U}_J(\ell)$ respectively. 
We recall the description of the action of Hecke operators on Jacobi forms 
in some special cases ({cf.} \cite[Ch. I \S 4]{EZ}, \cite[\S3]{MRVSK}). 
Let
$\phi(\tau,z) = \sum_{r,n} c(n,r) q^n \zeta^r$ be a form in $J^\mathrm{cusp}_{k,S}(T,\chi)$.
Write $(\phi|\mathrm{T}_J({m}))(\tau,z) = \sum_{r,n} c^*(n,r) q^n \zeta^r$. 
If $m$ is coprime with $M$, we have ({cf.} \cite[Thm. 4.5]{EZ})
\[c^*(n,r)=\sum_{d\mid m}\left(\frac{D}{d}\right)d^{k-2}c\left(\frac{m^2n}{d^2},\frac{mr}{d}\right).\]
Moreover, if $m=\ell$ is a prime number dividing $T$, then 

For any ring $R$, let $\mathcal P_{k-2}(R)$ denote
the $R$-module of homogeneous polynomials in 2 variables 
of degree $k-2$ with coefficients in $R$, 
equipped with the right action of the semigroup $\M_2(R)$ defined by the formula 
\begin{equation}\label{action}
(F|\gamma)(X,Y)=F\left(aX+bY,cX+dY\right)\end{equation}
for $\gamma=\smallmat abcd \in \M_2(R)$. Let $\mathcal{V}_{k-2}(R)$ denote the $R$-linear dual of 
$\mathcal{P}_{k-2}(R)$, equipped with the left $\M_2(R)$-action induced from the action on 
$\mathcal{P}_{k-2}(R)$. 

For each integer $\Delta$, let 
$\mathcal{Q}_{\Delta}$ be the set of integral quadratic 
forms  \[Q=[a,b,c]=ax^2+bxy+cy^2\]
of discriminant $\Delta$ and, for any integer $M\geq 1$ and any integer $\rho$, let
$\mathcal{Q}_{M,\Delta,\rho}$ denote the subset of $\mathcal{Q}_\Delta$ consisting 
of integral 
binary quadratic forms $Q=[a,b,c]$ of discriminant $\Delta$ such that 
$b\equiv\rho\pmod{2M}$ and $a\equiv 0\mod M$. Let 
$\mathcal{Q}^0_{M,\Delta,\rho}$ be the subset of 
$\mathcal{Q}_{M,\Delta,\rho}$ consisting of forms which are $\Gamma_0(M)$-primitive, 
i.e., those $Q\in \mathcal{Q}_{M,\Delta,\rho}$ 
which can be written as $Q=[Ma,b,c]$ with $(a,b,c)=1$. 
Those sets are equipped with the natural right action 
of $\SL_2(\Z)$ described in Equation \eqref{action}. 

Let $Q\mapsto\chi_{D_0}(Q)$ be the generalized genus character attached to $D_0$
defined in \cite[Ch. I, Sec. 1]{GKZ}. The character 
$\chi_{D_0}:\mathcal{Q}_{M,\Delta,\rho}\rightarrow\{\pm1\}$ depends on $(M,\Delta,\rho)$, 
and sometimes we will denote it by $\chi_{D_0}^{(M,\Delta,\rho)}$ to stress this dependency, 
or simply $\chi_{D_0}$ as above if $(M,\Delta,\rho)$ is understood.  
We recall the characterization of this 
character given in \cite[Prop. 1]{GKZ}.
If $Q=\ell\cdot Q'$ for some form $Q'\in\mathcal{Q}^0_{M,\Delta,\rho}$, 
then we define it as $\chi_{D_0}(Q)=\left(\frac{D_0}{\ell}\right)\cdot \chi_{D_0}(Q')$, so it is enough 
to define it on $\Gamma_0(M)$-primitive forms. Fix $Q\in\mathcal{Q}^0_{M,\Delta,\rho}$.  
If $(a/M,b,c,D_0)=1$, then pick any factorization $M={S}\cdot {T}$ with ${S}>0$, ${T}>0$ 
and any integer $n$ coprime with $D_0$ 
represented by the quadratic form $[a/{S},b,c{T}]$: then 
$\chi_{D_0}(Q):=\left(\frac{D_0}{n}\right)$. Otherwise i.e., if $(a/M,b,c,D_0)\neq1$, then
we set $\chi_{D_0}(Q):=0$. 

For $Q\in\mathcal{Q}_{M,\Delta,\rho}$, denote by
$C_Q$ the oriented geodesic path in the upper half plane attached to $Q$ 
joining the points $r_Q$ and $s_Q$ in $\PP^1(\Q)$, whose construction is 
described in \cite[\S2.1]{St}. 

Let $\mathcal D_0=\mathrm{Div}^0(\mathbb P^1(\Q))$
be the group of degree-zero divisors on rational cusps of the complex upper half plane 
$\mathcal H$, equipped with the left action of the semigroup {$\M_2(\Q)$} by fractional linear 
transformations. 
If $A$ is a $\Z[\M_2(R)]$-module 
the group of homomorphisms 
$\Hom(\mathcal{D}_0,A)$ is equipped with a standard left action of $\M_2(R)$ 
defined for $\gamma\in\M_2(R)$ and $\phi\in\Hom(\mathcal{D}_0,A)$ by 
$(\gamma\cdot\phi)(d)=\gamma\cdot\phi(d)$. 
If $\Gamma$ is a congruence subgroup of $\SL_2(\Z)$ and $A$ is a $\Z[\Gamma]$-module as above, 
let \[\mathrm{Symb}_\Gamma(A)=\Hom_\Gamma(\mathcal D_0,A)\] be the 
group of $A$-valued $\Gamma$-invariant modular symbols.  
As a general notation, if $m$ is a modular symbol in $\mathrm{Symb}_\Gamma(A)$, we write $m\{r\rightarrow s\}\in A$ for the
value of $m$ on the divisor $s-r\in\mathcal{D}_0$.  

We now recall the construction of 
the theta lifting of the modular form $f$ to the 
complex vector space of Jacobi forms, following \cite[Ch. II]{GKZ}.

Fix an integer $M\geq 1$, a divisor ${T}\mid M$, and 
a 
Dirichlet character 
\[\chi:(\Z/{T}\Z)^\times\rightarrow\C^\times.\] 
Put ${S}=M/{T}$, and assume that $({S},{T})=1$.   
Let $f$ be
a normalized cuspform of positive even weight $2k$ and level $\Gamma_0(M)$, 
where $M\geq 1$ is an odd integer, and character $\chi^2$.
To $f$ we may associate the modular symbol 
$\tilde{I}_f\in\mathrm{Symb}_{\Gamma_1(M)}(\mathcal{V}_{2k-2}(\C))$ by the integration formula 
\[\tilde{I}_f\{r\rightarrow s\}(P)=2\pi i\int_r^sf(z)P(z,1)dz.\] 

\begin{definition}\label{index pairpair}
A  {index $S$ pair} is a pair $(D,r)$ of integers 
consisting of a negative discriminant $D$ of an integral quadratic form 
$Q=[a,b,c]$ such that 
$D\equiv r^2\mod 4S$. An index $S$ pair 
is said to be  {fundamental} if $D$ is a fundamental discriminant. 
\end{definition}

Fix a fundamental index $S$ pair $(D_0,r_0)$. 
For any index $S$ pair $(D,r)$ let 
$\mathcal F_{D_0,r_0}^{(D,r)}({S},{T})$ be the set of 
integral binary quadratic forms 
$Q=[a,b,c]$ 
modulo the right action of $\Gamma_0(M)$
described in Equation \eqref{action},
such that: 
\begin{itemize} 
\item $\delta_Q=b^2-4ac={T}^2D_0D$; 
\item $b\equiv -{T}r_0r\mod 2{S}$;
\item $a\equiv 0\mod {S}{T}^2$. 
\end{itemize} 
If ${T}=1$, we simply 
write $\mathcal F_{D_0,r_0}^{(D,r)}(M)$ for $\mathcal F_{D_0,r_0}^{(D,r)}(M,1)$. 
In the previous notation, if $Q$ is a representative 
of a class in 
$\mathcal F_{D_0,r_0}^{(D,r)}({S},{T})$, then $Q$ belongs to 
$\mathcal{Q}_{{S},{T}^2D_0D,-{T}r_0r}$. In particular, we may consider the 
genus character $Q\mapsto\chi_{D_0}(Q)$ with 
$\chi_{D_0}=\chi_{D_0}^{(S,T^2D_0D,-Tr_0r)}$, 
for all classes $Q$ 
in $\mathcal F_{D_0,r_0}^{(D,r)}({S},{T})$. We may also define a 
character $Q=[a,b,c]\mapsto\chi(Q)=\chi(c)$ on the set 
$\mathcal F_{D_0,r_0}^{(D,r)}({S},{T})$. 

For this paper, we only need to consider the theta liftings studied in \cite{MR} in 
two particular situations: when $\chi$ is primitive and when $\chi$ is trivial and $T$ is a prime number. 
We make the definition in the two cases separately. 
If $\chi$ is primitive, define 
\[
\begin{split}
\tilde{I}_f(D,r)
&=\sum_{Q\in\mathcal F_{D_0,r_0}^{(D,r)}({S},{T})}\chi(Q)\cdot\chi_{D_0}(Q)\cdot\int_{\gamma_Q}f(z)Q(z,1)^{\frac{k-2}{2}}dz\\&
:=\sum_{Q\in\mathcal F_{D_0,r_0}^{(D,r)}({S},{T})}
\chi(Q)\cdot \chi_{D_0}(Q)\cdot \tilde{I}_f\{r_Q\rightarrow s_Q\}(Q^{k-1})
\end{split}
\]
while if $\chi$ is trivial, and $T$ is a prime number, define 
\[\begin{split}
\tilde{I}_f(D,r):=T^{k-1}\sum_{Q\in\mathcal F_{D_0,r_0}^{(D,r)}({S},{T})}&
\chi(Q)\cdot \chi_{D_0}(Q)\cdot \tilde{I}_f\{r_Q\rightarrow s_Q\}(Q^{k-1})\\
&-\left(\frac{D_0}{T}\right)\cdot\sum_{Q\in\mathcal F_{D_0,r_0}^{(T^2D,Tr)}({ST},{1}) }\chi(Q)\cdot \chi_{D_0}(Q^{k-1})\cdot \tilde{I}_f\{r_Q\rightarrow s_Q\}(Q)\end{split}
\]

\begin{remark} The definition in \cite{MR} when $\chi$ is trivial looks 
slightly different, since the second sum 
is taken 
over quadratic forms modulo $\Gamma_0(Np^2)$; however, 
each $\gamma\in \Gamma_0(Np)$, $\gamma\not\in \Gamma_0(Np^2)$ gives 
the same contribution to the above sum
and we get the formula displayed above
by taking into account the different powers of $T$ used here and in \cite{MR}.
\end{remark}

For $D=r^2 - 4{S}n$, put $\tilde{c}_f(n,r)=\tilde{I}_f(D,r)$. 
Then 
$\tilde{\mathcal{S}}_{D_0,r_0}^{(\chi)}(f)=\sum_{r,n} \tilde{c}_f(n,r) q^n \zeta^r$
belongs to $J_{k+1,{S}}^\mathrm{cusp}({T},\chi)$. The map 
$f\mapsto \tilde{\mathcal{S}}^{(\chi)}_{D_0,r_0}(f)$ from 
$S_{2k}(\Gamma_0(M),\chi^2)$ to 
$J_{k+1,{S}}^\mathrm{cusp}({T},\chi)$ thus obtained is called
the  {$(D_0,r_0)$-theta lifting} 
to the space of Jacobi forms ({cf.} \cite[Ch. II]{EZ}, \cite[Sec. 3]{MR}).
The matrix $\smallmat 100{-1}$ normalizes 
$\Gamma_1(M)$ and hence induces an involution on the space of modular symbols 
$\mathrm{Symb}_{\Gamma_1(M)}(\mathcal{V}_{2k-2}(\C))$; for each $\varepsilon\in\{\pm\}$, 
we denote by $\tilde{I}_f^\varepsilon$ the 
$\varepsilon$-eigencomponents of $\tilde{I}_f$ with respect to this involution. 
It is known that there are complex periods $\Omega_f^\varepsilon$ such that 
the $I_f^\varepsilon=\frac{\tilde{I}_f^\varepsilon}{\Omega_f^\varepsilon}$ belong 
to $\mathrm{Symb}_{\Gamma_1(M)}(\mathcal{V}_{2k-2}(F_f))$, where 
$F_f$ is the extension of $\Q$ generated by the Fourier coefficients of $f$. These periods 
can be chosen so that the Petersson norm $\langle f,f\rangle$ equals the product 
$\Omega_f^+\cdot\Omega_f^-$; note that the $\Omega_f^\varepsilon$ are well-defined 
only up to multiplication by non-zero factors in $F^\times_f$. 

Let 
$Q\in\mathcal F_{D_0,r_0}^{(D,r)}({S},{T})$. 
It follows easily from the definition of the generalized 
genus character that $\chi_{D_0}(Q)=\chi_{D_0}(Q|\iota)$
(it is enough to note that if the integer $n$ is represented 
by the quadratic form $[a/{S},b,c{T}]$ with the integers $x$ and 
$y$, then it is represented by the quadratic form 
$[a/{S},-b,c{T}]$ with the integers $x$ and $-y$). 
Also, one can easily check that $r_Q=-s_{Q|\iota}$ and $s_Q=-r_{Q|\iota}$, and therefore  
\[\left((\tilde{I}_f|\iota)\{r_Q\rightarrow s_Q\}\right)(Q^{k-1})=
\tilde{I}_f\{-r_Q\rightarrow -s_Q\}\left((Q|\iota)^{k-1}\right)=-
\tilde{I}_f\{r_{Q|\iota}\rightarrow s_{Q|\iota}\}\left((Q|\iota)^{k-1}\right).
\]
Finally, it is immediate that $\chi(Q)=\chi(Q|\iota)$. 
Combining these observations, we find that  
\[
\tilde{I}_f(D,r)
=
\sum_{Q\in\mathcal F_{D_0,r_0}^{(D,r)}({S},{T})}\chi(Q)\cdot\chi_{D_0}(Q)\cdot
\tilde{I}_f^-\{r_Q\rightarrow s_Q\}(Q^{k-1}).
\] 
Normalize the theta lifting by 
$\mathcal{S}^{(\chi)}_{D_0,r_0}(f)=\frac{\tilde{\mathcal{S}}^{(\chi)}_{D_0,r_0}(f)}{\Omega_{f}^-}$. 
We thus obtain a map 
\[\mathcal{S}^{(\chi)}_{D_0,r_0} : S_{2k}(\Gamma_0(M),\chi^2)\longrightarrow J^\mathrm{cusp}_{k+1,{S}}({T},\chi)\]
such that, if we write out the Fourier expansion
\[\mathcal{S}^{(\chi)}_{D_0,r_0}(f)=\sum_{r^2 - 4{S}n<0}{c}_f(n,r) q^n \zeta^r\]
then for $D=r^2 - 4{S}n$ we have 
\begin{equation}\label{Shintani}
{c}_f(n,r)= 
\sum_{Q\in\mathcal F_{D_0,r_0}^{(D,r)}({S},{T})}\chi(Q)\cdot\chi_{D_0}(Q)\cdot
{I}_f^-\{r_Q\rightarrow s_Q\}(Q^{k-1}).\end{equation}
In particular, if the Fourier coefficients of $f$ belong to a certain ring $\mathcal{O}\subseteq\bar\Q$,
then the same is true for the Fourier-Jacobi coefficients of 
$\mathcal{S}^{(\chi)}_{D_0,r_0}(f)$. 
If $\chi=\mathbf{1}$ is the trivial character, then we will simply write 
$\mathcal{S}_{D_0,r_0}$ for 
$\mathcal{S}_{D_0,r_0}^{(\mathbf{1})}$.
The map $\mathcal S_{D_0,r_0}^{(\chi)}$ are 
equivariant with respect to the action of Hecke operators i.e., 
$\mathcal S_{D_0,r_0}^{(\chi)}(f|\mathrm{T}(m))=\mathcal S_{D_0,r_0}^{(\chi)}(f)|\mathrm{T}_J(m)$. 

\subsection{Universal ordinary modular symbols} 
Recall that $\mathcal{O}$ is the valuation ring of a finite field extension of $\Q_p$. 
Let $\mathrm{Cont}(\Z_p^2,\mathcal{O})$ be the $\mathcal{O}$-module of $\mathcal{O}$-valued 
continuous functions on $\Z_p^2$, and let  
$\mathrm{Step}(\Z_p^2,\mathcal{O})$ be the $\mathcal{O}$-submodule of 
$\mathrm{Cont}(\Z_p^2,\mathcal{O})$ consisting of locally constant functions. 
Let \[\tilde{\D}=\Hom_{\Z_p}(\mathrm{Step}(\Z_p^2),\mathcal{O})\] be the group of 
$\mathcal{O}$-valued measures on $\Z_p^2$; we can extend in a unique way any $\mu\in\tilde{\D}$ to 
a function $\mu:\mathrm{Cont}(\Z_p^2,\mathcal{O})\rightarrow\mathcal{O}$. 
Adopting a standard convention, for $\varphi$ a continuous function on $\Z_p^2$, we denote 
$\int_{\Z_p^2}\varphi(x)d\mu$ 
the value $\mu(\varphi)$; if $\chi_X$ is the characteristic function of $X\subseteq\Z_p^2$, 
we write $\int_X\varphi(x)d\mu(x)$ for $\int_{\Z_p^2}\varphi(x)\cdot\chi_X(x)d\mu(x)$. 
Let $\mathbb D$ denote the $\mathcal{O}$-submodule of $\tilde{\D}$ consisting 
of those $\mathcal{O}$-valued measures which are supported on the 
set of primitive vectors $\mathbb X=(\Z_p^2)'$ in $\Z_p^2$. The $\mathcal{O}$-modules 
$\tilde{\D}$ and $\D$ are equipped with the action induced by the action of the
group $\GL_2(\Z_p)$ on $\Z_p^2$ by $(x,y)\mapsto (ax+by,cx+dy)$ for $\gamma=\smallmat abcd\in\GL_2(\Z_p)$. The $\mathcal{O}$-module $\mathbb{D}$ is also 
equipped with a structure of $\mathcal{O}\pwseries{\Z_p^\times}$-module induced by the scalar action of $\Z_p^\times$ on $\mathbb X$ ({cf.} \cite[\S 5]{St}, \cite[\S 2.2]{BD}); in particular, $\D$ is also equipped 
with a structure of $\Lambda$-module.   
Let $\Gamma_0(p\Z_p)$ and $\Gamma_1(p\Z_p)$ denote the subgroups of $\GL_2(\Z_p)$ 
consisting of matrices 
which are respectively 
upper triangular and congruent to a matrix of the form $\smallmat 1{*}01$ modulo $p$. 

The group $\mathrm{Symb}_{\Gamma_1(N)}(\mathbb{D})$ is equipped with an action of 
Hecke operators, and we denote 
by \[\W=\mathrm{Symb}_{\Gamma_1(N)}^\ord(\mathbb{D})\] the 
ordinary subspace of $\mathrm{Symb}_{\Gamma_1(N)}(\mathbb{D})$ for the action of the 
$\U_p$-operator; see \cite[(2.2), (2.3)]{GS} for details and more accurate definitions. 
For any $\Lambda$-algebra $R$, write $\W_R=\W\otimes_\Lambda{R}$. 

We fix as in \S \ref{sec:HidaFamilies} 
an integer $N\geq 1$, a prime number $p\nmid N$, $p\geq5$, a primitive branch $\mathcal{R}$ of the Big Hida Hecke algebra 
acting on modular forms of tame level $\Gamma_1(N)$ and a primitive Hida familiy $f_\infty(\kappa)=\sum_{n\geq 1}a_n(\kappa)q^n\in \mathcal{R}\pwseries{q}$. 
If $\kappa$ is an arithmetic point of $\mathcal{X}(\mathcal{R})$ 
of signature $(\chi,k)$, and wild level $p^s$, 
then we  
have a $\Gamma_s$-equivariant homomorphism 
$\rho_\kappa:\mathbb{D}\rightarrow\mathcal{V}_{k-2}(\C_p)$ defined by the formula 
\[\rho_\kappa(\mu)(P)=\int_{\Z_p\times\Z_p^\times}\chi(y)\cdot P(x,y)d\mu(x,y).\]
The homomorphism $\rho_\kappa$ gives rise to an homomorphism, denoted by the same
symbol,  
\[\rho_\kappa:\W_{\mathcal{R}}\longrightarrow 
\mathrm{Symb}_{\Gamma_1(N)}\left(\mathcal{V}_{k-2}(\C_p)\right).\]
For $\Phi\in \W_{\mathcal{R}}$, we put 
$\Phi_{\kappa}=\rho_\kappa(\Phi)$. 

Fix an arithmetic point $\kappa_0$. By \cite[Thm. 5.5]{St}, there exists 
$\Phi_{\kappa_0} \in\mathrm{Symb}_{\Gamma_1(N)}(\mathbb{D}_\mathcal{R})$ 
such that 
\begin{equation}
\label{GS}
\Phi_{\kappa_0}(\kappa)=\lambda(\kappa)\cdot I^-_{f_{\kappa}}\end{equation}
with {$\lambda(\kappa)\in \calR_{\kappa}$ such that $\lambda(\kappa_0)=1$.} 
When $\kappa_0$ is understood, we will simply write $\Phi$ for $\Phi_{\kappa_0}$. 

\subsection{$\Lambda$-adic liftings}\label{sec3.4}

For any $\Lambda$-algebra $R$, we write $\D_R=\D\otimes_\Lambda R$. We switch back to the notation used for Hida families over $\GL_2$.
Recall that $\widetilde{\mathcal X}(\mathcal R)$ is the $\mathcal O$-module of continuous   $\mathcal O$-linear homomorphisms from $\widetilde{\mathcal R}$ to $\bar\Q_p$, where $\widetilde{\mathcal R}=\mathcal R\otimes_\Lambda\widetilde\Lambda$ and  the $\Lambda$-algebra $\widetilde\Lambda$ is equal to $\Lambda$ as an $\mathcal O$-module, but the structure of $\Lambda$-algebra is given by the map $\sigma(t)= t^2$ on $1+p\Z_p$. We have the notion of  arithmetic points $\widetilde{\mathcal X}^\mathrm{arith}(\mathcal R)$ and that of  signature $(\chi,k)$ of an arithmetic point $\tilde\kappa$ in $\widetilde{\mathcal X}^\mathrm{arith}(\mathcal R)$; the effect of $\pi$ on an arithmetic point $\tilde \kappa$ is to double the signature: if $\tilde\kappa$ has signature $(\chi,k)$, then $\pi(\tilde\kappa)$ has signature $(\chi^2,2k)$.

The following result is essentially \cite[Lemma (6.1)]{St}, with minor modifications which are left to the reader.  

\begin{lemma}\label{lemmastevens} Fix a positive integer $k_0$, 
a positive integer $s$, 
and a quadratic form $Q=[a,b,c]$ such that $p^s\mid a$, $p^s\mid b$,  
$p\nmid c$.  
There exists a unique 
morphism 
$j_{Q,k_0}:\mathbb D_\mathcal R\rightarrow \widetilde{\mathcal R}$ of $\mathcal R$-modules 
characterized by the following property: For any pair of points $\kappa\in\mathcal{X}^\mathrm{arith}(\mathcal{R})$ 
and $\tilde\kappa\in\mathcal{X}^\mathrm{arith}(\widetilde{\mathcal{R}})$ of signature 
$(\chi^2,2k)$ and $(\chi,k)$ respectively, satisfying the conditions 
\begin{enumerate}
\item $\pi(\tilde\kappa)=\kappa$; 
\item $2k\equiv 2k_0\pmod{p-1}$; 
\item the wild level of $\chi$ is $p^s$;  
\end{enumerate} 
we have $
\tilde\kappa\left(j_{Q,k_0}(\mu)\right)=\chi(Q)\cdot\rho_\kappa(\mu)\left(Q^{k-1}\right)$. 
\end{lemma}

We now prove a couple of auxiliary lemmas on quadratic forms. We denote by $\mathcal{P}_{D_0,r_0}^{(D,r)}(N,p^s)$ the subset of $\mathcal{F}_{D_0,r_0}^{(D,r)}(N,p^s)$ 
consisting of classes represented by quadratic forms $[a,b,c]$ 
with $p\nmid c$.  Note that if $s\geq 1$, and $[a,b,c]\in\mathcal{P}_{D_0,r_0}^{(D,r)}(N,p^s)$, then 
$p^s\mid b$: this is because 
$p^{2s}\mid a$ and $p^{2s}\mid b^2-4ac$, which implies that $p^s\mid b$.  

\begin{lemma}\label{rep3}
The map 
$\mathcal{P}_{D_0,r_0}^{(D,r)}(N,p^s)\rightarrow 
\mathcal{P}_{D_0,r_0}^{(Dp^{2(s-1)},rp^{s-1})}(N,p)$ 
which takes $[a,b,c]\mapsto[a,b,c]$ is a bijection. 
\end{lemma} 

\begin{proof} The map is clearly well defined. 

We first show the surjectivity. For this, 
we need to show that
we may choose a representative $[a,b,c]$ for any class 
in $\mathcal{P}_{D_0,r_0}^{(Dp^{2(s-1)},rp^{s-1})}(N,p)$
so that $p^{2s}\mid a$, and for this it is enough to 
show that $p^{s}\mid b$ because $p^{2s}\mid\delta_Q$
and $p\nmid c$. 
So fix $Q=[a,b,c]$ a representative of a class in 
$\mathcal{P}_{D_0,r_0}^{(Dp^{2(s-1)},rp^{s-1})}(N,p)$. 
Let $a\in\Z$ such that 
$a\equiv -b/2Ncp\pmod{p^s}$ (note that $p\mid b$ 
because $p^{2s}\mid b^2-4ac$ and $p^2\mid a$, so, 
since $p\nmid c$, 
$-b/2cp$ is a $p$-adic integer)   
and consider the matrix $\gamma=\smallmat10{Nap}1$. 
Then $Q|\gamma=[a'b'c']$ with $b'=b+2Ncpa\equiv0\pmod{p^s}$. 
Since $\gamma\in\Gamma_0(Np)$, 
this shows that we can change representatives of 
any element in $\mathcal{P}_{D_0,r_0}^{(Dp^{2(s-1)},rp^{s-1})}(N,p)$
so that $p^{2s}\mid a$, and therefore this is 
an element in $\mathcal{P}_{D_0,r_0}^{(D,r)}(N,p^s)$, 
proving the surjectivity. We now show that the above map is 
injective. Take two quadratic forms
$[a,b,c]$ and $[a',b',c']$ representing 
classes in $\mathcal{P}_{D,0,r_0}^{(D,r)}(N,p^s)$ and assume they are $\Gamma_0(N)$-equivalent. 
Recall that $p\nmid cc'$.   
Since a set of representatives of $\Gamma_0(N)$ modulo $\Gamma_0(Np^s)$ 
is given by the matrices $\gamma_i=\smallmat10{iN}1$ for $i=0,\dots,p^s-1$, 
we see that, up to changing the representative $[a',b',c']$, there exists $0\leq i<p^s$
such that $[a',b',c']=[a,b,c]|\gamma_i$, and therefore $a'=a+biN+ci^2N^2$. 
Since $p^{2s}\mid a$ and $p^{2s}\mid a'$, 
we then see that  $p^{2s}\mid iN(b+ciN)$. Suppose now $i\neq 0$, and 
write $t=\ord_p(i)$. Then $p^{2s-t}\mid b+ciN$ 
because $p\nmid N$. So $b\equiv -ciN\pmod{p^{2s-t}}$, 
and therefore $b^2\equiv c^2(iN)^2\pmod{p^{2s}}$. On the other hand, 
$p^{s}\mid b$, and therefore, since $p\nmid Nc$, we get $p^s\mid i$, which is not possible. 
This proves that $i=0$, and therefore $[a,b,c]$ and $[a',b',c']$ 
are $\Gamma_0(Np^s)$-equivalent, proving the injectivity of the map. 
\end{proof}

Recall the fixed arithmetic point   
$\kappa_0\in\mathcal{X}^\mathrm{arith}(\mathcal{R})$
of signature $(\chi_0^2,2k_0)$, and the modular symbol 
$\Phi_{\kappa_0}$ attached to $\kappa_0$ in Equation \eqref{GS}.
For $Q$ in $\mathcal{F}_{D_0,r_0}^{(D,r)}{(N,p)}$ put $D_Q=s_Q-r_Q$.
Let   
$\tilde\kappa_0\in\mathcal{X}^\mathrm{arith}(\widetilde{\mathcal{R}})$ be such 
that the pair $(\kappa_0,\tilde\kappa_0)$ satisfies the condition of Lemma \ref{lemmastevens}. 
Extend $j_{Q,k_0}$ to a map $j_{Q,k_0}:
\D_{\mathcal{R}_{\kappa_0}}\rightarrow\widetilde{\mathcal{R}}_{\tilde\kappa_0}$. 
Fix a fundamental index $N$ pair $(D_0,r_0)$. 
For any index $N$ pair $(D,r)$, define  
\begin{equation}\label{familycoeff}
c_{n,r} :=\sum_{Q\in\mathcal{P}_{D_0,r_0}^{(D,r)}(N,p)} \chi_{D_0}(Q)\cdot j_{Q,k_0}\left(\Phi_{\kappa_0}(D_Q)\right).\end{equation} 
Here $\chi_{D_0}=\chi_{D_0}^{(N,DD_0p^2,-rr_0p)}$. 

\begin{dfn}\label{def-fam}
The {$(D_0,r_0)$-family of Shintani lifts centered at $\kappa_0$}
is the formal power series in 
$\mathbb{S}_{D_0,r_0}^{(\kappa_0)}$ in $\widetilde{{\mathcal R}}_{\kappa_0}\pwseries{q,\zeta}$ defined by 
\begin{equation}\label{lift}
\mathbb{S}_{D_0,r_0}^{(\kappa_0)}=
\sum_{D=r^2-4Nn<0}
c_{n,r}q^n\zeta^r.
\end{equation} 
\end{dfn}

When $\kappa_0$ is fixed, we will simply write $\mathbb{S}_{D_0,r_0}$ in place 
of $\mathbb{S}_{D_0,r_0}^{(\kappa_0)}$ to lighten the notation. 

\begin{remark}In the definition above the coefficients $c_{n,r}$ in Equation \eqref{familycoeff} {depend} on the choice of $D_0,r_0$ and $\kappa_0$, but we do not make explicit this dependence to make the notation simpler; similar remarks will apply to similarly defined coefficients. However, we keep here and in the following the 
dependence on these quantities in the notation for the $\Lambda$-adic families themselves, 
such as $\mathbb{S}_{D_0,r_0}^{(\kappa_0)}$.\end{remark}

Let $\kappa_0$ be fixed from now on, and write $\Phi$ and
$\mathbb{S}_{D_0,r_0}$ for $\Phi_{\kappa_0}$ and $\mathbb{S}_{D_0,r_0}^{(\kappa_0)}$.
For any $\tilde\kappa\in\mathcal{X}(\widetilde{\mathcal{R}})$ write 
$\mathbb{S}_{D_0,r_0}(\tilde\kappa)=\tilde\kappa\left(\mathbb{S}_{D_0,r_0}\right)$ for the specialization at $\tilde\kappa$. As the name suggests, $\mathbb{S}_{D_0,r_0}^{(\kappa_0)}$ in Definition \ref{def-fam} interpolates 
in families Shintani lifts of classical forms in Hida families, as the following theorem shows.

\begin{theorem}\label{MainFamilies}
Let $\kappa\in\mathcal{X}^\mathrm{arith}(\mathcal{R})$ 
be an arithmetic point of signature $(\chi^2,2k)$ 
and let $\tilde\kappa\in\widetilde{\mathcal{X}}^\mathrm{arith}(\mathcal{R})$ be an arithmetic point of signature $(\chi,k)$ be 
such that $\pi(\tilde\kappa)=\kappa$. Let $\chi=\omega^{k-1}\psi$ be the character of $\tilde\kappa$, 
so that $\chi^2$ is the character of $\kappa$, and assume that the conductor of 
$\chi$ is $p^s$ for some integer $s\geq 1$. 
Then 
\[\mathbb{S}_{D_0,r_0}(\tilde\kappa)=\lambda(\kappa)\cdot\mathcal{S}_{D_0,r_0}^{(\chi)}(f_\kappa)|\mathrm{T}_J({p})^{1-s}.\]
\end{theorem}

\begin{proof} 
We first observe that, 
combining Lemma \ref{lemmastevens} and Equation \eqref{GS}, for all $Q\in \mathcal{P}_{D_0,r_0}^{(D,r)}(N,p^s)$ 
we have 
\[ \tilde\kappa\left(j_{Q,k_0}\left(\Phi_{\kappa_0}(D_Q)\right)\right)=
\chi(Q)\cdot\Phi_{\kappa_0}(D_Q)(k)\left(Q^{k-1}\right)=
\lambda(k)\cdot \chi(Q)\cdot I^-_{f_{\kappa}}\{r_Q\rightarrow s_Q\}\left(Q^{k-1}\right).
\] 

Next, note that if the quadratic form $Q=[a,b,c]$
has $c=0$, then $\chi(Q)=0$ and 
the summand corresponding to $Q$ does not appear in 
Equation \eqref{Shintani}. Therefore, the sum in Equation \eqref{Shintani} apparently over $\mathcal{F}_{D_0,r_0}^{(D,r)}({N},{p^s})$
is really only over
$\mathcal{P}_{D_0,r_0}^{(D,r)}(N,p^s)$. 

We now prove the theorem. If $s=1$, the term $\mathrm{T}_J({p})^{1-s}$ disappears so this is an immediate consequence of the previous 
two observations, so let us assume that $s>1$. 

For
$F=\sum_{n,r}a_{n,r}q^n\zeta^r\in A[\![q,\zeta]\!]$, where $A$ is any 
ring, and $m\geq0$ an integer, define the formal power series 
\[(F|\mathrm{T}_J^m(p))(q,\zeta)=\sum_{n,r}a_{p^mn,p^mr}q^n\zeta^r.\] The 
formulas recalled in Section \ref{sec-theta} show that if $F$ is a Jacobi form
in $J^\mathrm{cusp}_{k+1,N}(p^s,\chi)$, then the operator just defined 
coincides with the Hecke operator $\mathrm{T}_J^m(p)$, justifying the abuse of notation.  
We have 
\[\mathbb{S}_{D_0,r_0}(\tilde\kappa)|\mathrm{T}_J^{s-1}(p)=\lambda(\kappa)\cdot
\sum_{Q\in\mathcal{P}_{D_0,r_0}^{(Dp^{2(s-1)},rp^{s-1})}(N,p)}\chi(Q)\cdot\chi_{D_0}(Q)\cdot I^-_{f_{\kappa}}\{r_Q\rightarrow s_Q\}\left(Q^{k-1}\right).\]
Lemma \ref{rep3} shows that there exists a common set of representatives for 
$\mathcal{P}_{D_0,r_0}^{(Dp^{2(s-1)},rp^{s-1})}(N,p)$ and 
$\mathcal{P}_{D_0,r_0}^{(D,r)}(N,p^s)$, and therefore the sum in the right hand side 
of the last displayed equality is 
$\lambda(k)\cdot\mathcal{S}_{D_0,r_0}^{(\chi)}(f_\kappa)$.
It might be also useful to notice that we may alternatively view $\chi_{D_0}$ 
as a function on $\mathcal{P}_{D_0,r_0}^{(Dp^{2(s-1)},rp^{s-1})}(N,p)$ or 
$\mathcal{P}_{D_0,r_0}^{(D,r)}(N,p^s)$, since in both cases 
$\chi_{D_0}$ is defined as $\chi_{D_0}^{(N,D_0Dp^{2s},-r_0rp^s)}$. 
We therefore get the equality 
of power series
\[\mathbb{S}_{D_0,r_0}(\tilde\kappa)|\mathrm{T}_J^{s-1}(p)=\lambda(\kappa)\cdot\mathcal{S}_{D_0,r_0}^{(\chi)}(f_\kappa).\] 
The Hecke operator $\mathrm{T}_J(p)$ acts on $\mathcal{S}_{D_0,r_0}^{(\chi)}(f_\kappa)$ 
as multiplication by the $p$-adic unit $a_p(\kappa)$, and the theorem follows then 
applying $\mathrm{T}_J^{1-s}(p)$ to the above formula.
\end{proof}

We now study the specialization of the family of Jacobi forms 
$\mathbb{S}_{D_0,r_0}$ to arithmetic points with trivial character.

\begin{theorem}\label{MainFamiliesTrCh} Let $\kappa\in\mathcal{X}^\mathrm{arith}(\mathcal{R})$ 
be an arithmetic point of signature $(\mathbf{1},2k)$, 
let $(D_0,r_0)$ be a fundamental index $N$ pair such that $p\nmid D_0$. 
Let $\tilde\kappa$ be 
an arithmetic point of $\tilde{\mathcal{X}}^\mathrm{arith}(\mathcal{R})$ 
such that $\pi(\tilde\kappa)=\kappa$. 
Then, for any index $N$ pair $(D,r)$ such that $p\nmid D$, satisfying the relation $D=r^2-4Nn$ for an integer $n$, 
we have 
\[c_{n,r}(\tilde\kappa)=\lambda(\kappa)\cdot
 c_{f_\kappa}(n,r),\] where $c_{f_\kappa}(n,r)$ is the 
Fourier-Jacobi coefficients of $\mathcal{S}_{D_0,r_0}^{(\mathbf{1})}(f_\kappa)$.
\end{theorem}

\begin{proof}
We first note that the result is clear if $\tilde\kappa$ has signature $(\chi,k)$ with 
$\chi\neq\mathbf{1}$, since then $\chi(Q)=0$ for all non-primitive forms. 
So assume that $\chi = {\mathbf{1}}$. 
As in the proof of Theorem \ref{MainFamilies}, combining Lemma \ref{lemmastevens} and Equation \eqref{GS} 
we get  
\[c_{n,r}(\tilde\kappa)= \lambda(\kappa)\cdot\sum_{Q\in\mathcal{P}_{D_0,r_0}^{(D,r)}(N,p)} \chi_{D_0}(Q)\cdot  I^-_{f_{\kappa}}\{r_Q\rightarrow s_Q\}\left(Q^{k-1}\right).\]
Since $p\nmid D_0D$, the 
set $\mathcal{F}_{D_0,r_0}^{(D,r)}(N,p)$ is the disjoint union of the two sets 
$\mathcal{P}_{D_0,r_0}^{(D,r)}(N,p)$ and the subset $\mathcal{N}_{D_0,r_0}^{(D,r)}(N,p)$ 
of  $\mathcal{F}_{D_0,r_0}^{(D,r)}(N,p)$ consisting of non-primitive forms.
Therefore, 
\[
\begin{split}
&\sum_{Q\in\mathcal{P}_{D_0,r_0}^{(D,r)}(N,p)}
\chi_{D_0}(Q)\cdot  I^-_{f_{\kappa}}\{r_{Q}\rightarrow s_{Q}\}\left(Q^{k-1}\right)=\\
& = \sum_{Q\in \mathcal{F}_{D_0,r_0}^{(D,r)}(N,p)}
\chi_{D_0}(Q)\cdot  I^-_{f_{\kappa}}\{r_{Q}\rightarrow s_{Q}\}\left(Q^{k-1}\right) - 
\sum_{Q\in \mathcal{N}_{D_0,r_0}^{(D,r)}(N,p)}
\chi_{D_0}(Q)\cdot  I^-_{f_{\kappa}}\{r_{Q}\rightarrow s_{Q}\}\left(Q^{k-1}\right).
\end{split}\]

Since $\chi_{D_0}(p\cdot Q)=\left(\frac{D_0}{p}\right)\cdot\chi_{D_0}(Q)$ if $Q$ is not primitive, and $r_{p\cdot Q}=r_Q$ and $s_{p\cdot Q}=s_Q$, the second
summand of the last displayed formula is equal to 
\[p^{k-1}\left(\frac{D_0}{p}\right)\cdot\sum_{Q\in\mathcal{S}_{D_0,r_0}^{(D,r)}(N,p)}
\chi_{D_0}(Q)\cdot  I^-_{f_{\kappa}}\{r_{Q}\rightarrow s_{Q}\}\left(Q^{k-1}\right).\]
where $\mathcal{S}_{D_0,r_0}^{(D,r)}(N,p)=\{Q/p:Q\in\mathcal{N}_{D_0,r_0}^{(D,r)}(N,p)\}$.
The set $\mathcal{S}_{D_0,r_0}^{(D,r)}(N,p)$ consists then of forms $[a,b,c]$ 
modulo $\Gamma_0(Np)$ 
satisfying the three conditions $b^2-4ac=D_0D$, $b\equiv -r_0r\pmod{N}$ and $a\equiv0\pmod{Np}$, 
and therefore this is $\mathcal{F}_{D_0,r_0}^{(D,r)}(Np,1)$.
Comparing with the 
definition of the theta correspondence in the case of trivial character gives the result. 
\end{proof}

We now relate the coefficients $c_{n,r}(\tilde\kappa)$ with the theta lift 
of oldforms appearing in the Hida families. More precisely, 
let $\kappa$ be an arithmetic point of even weight $2k>2$ and trivial character. 
The modular form $f_{\kappa}$ is then $p$-old, and there exists a newform 
$f_\kappa^\sharp$ of level $\Gamma_0(N)$ and trivial character 
whose $p$-stabilization is $f_{\kappa}$; in other words, we have the relation 
\[f_\kappa(z)=f_\kappa^\sharp(z)-\frac{p^{2k-1}}{a_p(\kappa)}f_\kappa^\sharp(pz).\]
If $k=2$ and the character is trivial, then either $f_\kappa$ is the 
$p$-stabilization of a form $f_\kappa^\sharp$ as above, or $f_\kappa$ 
is a newform of level $\Gamma_0(Np)$. 
We need a couple of technical lemmas on quadratic forms. 

\begin{lemma}\label{rep1}
Let $(D_0,r_0)$ be a fixed fundamental index $N$ pair. For any index $N$ pair $(D,r)$, there exists a system of representatives $\mathcal{R}_{D_0,r_0}^{(Dp^2,rp)}(N)$ for 
$\mathcal{P}_{D_0,r_0}^{(Dp^2,rp)}(N)$ consisting of forms $[a,b,c]$ 
with $p\nmid c$ and $p\mid b$. 
\end{lemma}

\begin{proof}
Fix a $p$-primitive form $Q=[a,b,c]$ representing a class in
$\mathcal{P}_{D_0,r_0}^{(D,r)}(N,p)$. We first claim 
that, up to a change of representatives, we can assume that $p\mid b$. 
To show this, note that if $p\nmid b$ then $p\nmid a$ because $p\mid=b^2-4ac$. 
Choose any integer $\beta$ such that $\beta\equiv -b/2a\pmod p$, 
and let $\gamma=\smallmat 1\beta01$. Then $Q|\beta=[a',n',c']$ 
with $p\mid c'$. So we may assume that the representative 
$Q=[a,b,c]$ is chosen so that $p\mid b$. 
If now $p\mid c$, then $p\mid a$, because otherwise 
$[a,b,c]$ is not $p$-primitive. Since $p\mid b$ and $p^2\mid b^2-4ac$, 
we see that $p^2\mid c$. Take two integers $\alpha$ and 
$\beta$ such that $-N\alpha+p^2\beta=1$ and consider the 
matrix $\gamma=\smallmat{p\beta}{\alpha}{N}{p}$. Then $Q|\gamma$ has 
the required property that $p\mid b$ and $p\nmid c$. 
\end{proof}

\begin{lemma}\label{rep2} Let $\mathcal{R}_{D_0,r_0}^{(Dp^2,rp)}(N)$ be fixed as in Lemma \ref{rep1}
and $(D_0,r_0)$ a fundamental index $N$ pair and $(D,r)$ an index $N$ pair. 
The canonical map $[a,b,c]\mapsto [a,b,c]$ induces 
a bijection  $\mathcal{P}_{D_0,r_0}^{(D,r)}(N,p)\rightarrow 
\mathcal{R}_{D_0,r_0}^{(Dp^{2},rp)}(N)$. 
\end{lemma} 

\begin{proof} 
The map is clearly well defined. We show that it is a bijection. We first show the surjectivity. It is clearly enough to show that each $p$-primitive 
form in the target comes from a (necessarily $p$-primitive) form in the source via the map 
in the statement. 
By Lemma \ref{rep1}
it is enough to show that each 
quadratic form $[a,b,c]$ with $p\nmid c$ and $p\mid b$, representing a class in $\mathcal{F}_{D_0,r_0}^{(Dp^{2},rp)}(N)$, also represents a class in $\mathcal{F}_{D_0,r_0}^{(D,r)}(N,p)$, namely, 
$p^{2}\mid a$. 
This is clear: the discriminant $b^2-4ac$ of the 
quadratic form $[a,b,c]$ is divisible by $p^{2}$, $b^2$ is divisible by $p^{2}$ and $p\nmid c$. 

We show the injectivity. Take two primitive quadratic forms
$[a,b,c]$ and $[a',b',c']$ representing 
classes in $\mathcal{F}_{D,0,r_0}^{(D,r)}(N,p)$ and assume they are $\Gamma_0(N)$-equivalent. 
Since these are $p$-primitive forms, $p\nmid cc'$.   
A set of representatives of $\Gamma_0(N)$ modulo $\Gamma_0(Np)$ 
is given by the matrices $\gamma_i=\smallmat10{iN}1$ for $i=0,\dots,p-1$. 
So, up to changing the representative $[a',b',c']$, there exists $0\leq i<p$
such that $[a',b',c']=[a,b,c]|\gamma_i$, and therefore $a'=a+biN+ci^2N^2$. 
Since $p^{2}\mid a$ and $p^{2}\mid a'$, 
we then see that  $p^{2}\mid iN(b+ciN)$. Suppose now $i\neq 0$. Then $p^{2}\mid b+ciN$ 
because $p\nmid N$. So $b\equiv -ciN\pmod{p^{2}}$. However,  
$p\mid b$ but $p\nmid Nc$, so we get $p\mid i$, which is not possible because $i<p$. 
This proves that $i=0$, and therefore $[a,b,c]$ and $[a',b',c']$ 
are $\Gamma_0(Np)$-equivalent, proving the injectivity of the map. 
\end{proof}

\begin{theorem}\label{new forms}
Assume that $f_\kappa$ is the $p$-stabilization of $f_\kappa^\sharp$, and 
let $(D_0,r_0)$ be a fundamental index $S$ pair. 
Let $\tilde\kappa$ be 
an arithmetic point of $\tilde{\mathcal{X}}^\mathrm{arith}(\mathcal{R})$ of signature $(\mathbf{1},k)$, 
such that $\pi(\tilde\kappa)=\kappa$, an arithmetic point in 
$\mathcal{X}^\mathrm{arith}(\mathcal{R})$ of signature $(\mathbf{1},2k)$. 
Then, for any index $N$ pair $(D,r)$ such that $p\nmid D$, satisfying the relation $D=r^2-4Nn$ for an integer $n$, 
 we have \[c_{n,r}(\tilde\kappa)= \lambda(\kappa)\cdot
a_p(\kappa)\left(1-
\frac{p^{k-1}\left(\left(\frac{D_0}{p}\right)+\left(\frac{D}{p}\right)\right)}{a_p(\kappa)}
-\frac{p^{2k-2}}{a_p^2(\kappa)}
\right)\cdot c_{f_\kappa^\sharp}(n,r),\]
where $c_{f_\kappa^\sharp}(n,r)$ is the 
Fourier-Jacobi coefficient of $\mathcal{S}_{D_0,r_0}^{(\mathbf{1})}(f_\kappa^\sharp)$. 

\end{theorem}

\begin{proof}
 
As in the proof of Theorem \ref{MainFamilies}, combining Lemma \ref{lemmastevens} and Equation \eqref{GS} 
we get  
\[c_{n,r}(\tilde\kappa)= \lambda(\kappa)\cdot\sum_{Q\in\mathcal{P}_{D_0,r_0}^{(D,r)}(N,p)} \chi_{D_0}(Q)\cdot  I^-_{f_{\kappa}}\{r_Q\rightarrow s_Q\}\left(Q^{k-1}\right).\]
By definition, 
\[\int_{r_Q}^{s_Q}f_\kappa(z)Q(z,1)^{k-1}dz=
\int_{r_Q}^{s_Q}f_\kappa^\sharp(z)Q(z,1)^{k-1}dz-
\frac{p^{2k-1}}{a_p(\kappa)}\int_{r_Q}^{s_Q}f_\kappa^\sharp(pz)Q(z,1)^{k-1}dz.\]
Changing variable $z\mapsto z/p$, we have  
\[\int_{r_Q}^{s_Q}f_\kappa^\sharp(pz)Q(z,1)^{k-1}dz=
\frac{1}{p}\cdot\int_{r_{Q_p}}^{s_{Q_p}}f_\kappa^\sharp(z)Q_p(z,1)^{k-1}dz\]
where $Q_p=[a/p^2,b/p,c]$. We therefore obtain the equality 
\begin{equation}\label{p-stab}
I^-_{f_{\kappa}}\{r_Q\rightarrow s_Q\}\left(Q^{k-1}\right)= 
I^-_{f_{\kappa}^\sharp}\{r_Q\rightarrow s_Q\}\left(Q^{k-1}\right)-
\frac{p^{2k-2}}{a_p(\kappa)}I^-_{f_{\kappa}^\sharp}\{r_{Q_p}\rightarrow s_{Q_p}\}\left(Q_p^{k-1}\right).
\end{equation}
To compute $c_{n,r}(\tilde\kappa)$ we need to take the sum over all forms 
$Q\in\mathcal{P}_{D_0,r_0}^{(D,r)}(N,p)$ in Equation \eqref{p-stab}. 

We begin by computing the sum over all $Q\in\mathcal{P}_{D_0,r_0}^{(D,r)}(N,p)$ 
of the first summand in the right hand side of Equation \eqref{p-stab}. 
Suppose first that $p\nmid D_0$. Then $p\nmid D_0D$ and the  
set $\mathcal{F}_{D_0,r_0}^{(Dp^2,rp)}(N)$ is the disjoint union of the two sets 
$\mathcal{P}_{D_0,r_0}^{(Dp^2,rp)}(N)$ and $p\cdot\mathcal{P}_{D_0,r_0}^{(D,r)}(N)$. 
Therefore, identifying 
$\mathcal{P}_{D_0,r_0}^{(D,r)}(N,p)$ with $\mathcal{R}_{D_0,r_0}^{(Dp^2,rp)}(N)$ as above, we see that 
\[
\sum_{Q\in\mathcal{P}_{D_0,r_0}^{(D,r)}(N,p)}
\chi_{D_0}(Q)\cdot  I^-_{f_{\kappa}^\sharp}\{r_{Q}\rightarrow s_{Q}\}\left(Q^{k-1}\right)=\]
\[=\sum_{Q\in\mathcal{F}_{D_0,r_0}^{(Dp^2,rp)}(N)}\chi_{D_0}(Q)\cdot  I^-_{f_{\kappa}^\sharp}\{r_{Q}\rightarrow s_{Q}\}\left(Q^{k-1}\right)-
\sum_{Q\in p\cdot\mathcal{P}_{D_0,r_0}^{(D,r)}(N)}
\chi_{D_0}(Q)\cdot  I^-_{f_{\kappa}^\sharp}\{r_{Q}\rightarrow s_{Q}\}\left(Q^{k-1}\right).\]
Since $\chi_{D_0}(p\cdot Q)=\left(\frac{D_0}{p}\right)\cdot\chi_{D_0}(Q)$ if $Q$ is primitive, and $r_{p\cdot Q}=r_Q$ and $s_{p\cdot Q}=s_Q$, the second
summand of the last displayed formula is equal to 
$p^{k-1}\left(\frac{D_0}{p}\right)c_{f_\kappa^\sharp}(n,r)$, 
and we get 
\[\sum_{Q\in\mathcal{P}_{D_0,r_0}^{(D,r)}(N,p)}
\chi_{D_0}(Q)\cdot  I^-_{f_{\kappa}^\sharp}\{r_{Q}\rightarrow s_{Q}\}\left(Q^{k-1}\right)=
c_{f_\kappa^\sharp}(np^2,rp)-p^{k-1}\left(\frac{D_0}{p}\right)c_{f_\kappa^\sharp}(n,r).\]
If $p\mid D_0$, 
then the above formula still holds, since in this case $\chi_{D_0}(Q)=0$ when $Q$ 
is not $p$-primitive, and therefore the sum in Equation \eqref{p-stab}
reduces to $c_{f_\kappa^\sharp}(np^2,rp)$. 
Finally, using the formulas for the Hecke operator $\mathrm{T}_J(p)$ in Sec. \ref{sec-theta}, 
we obtain the relation 
\[a_p(\kappa)\cdot c_{f_\kappa^\sharp}(n,r)=
c_{f_\kappa^\sharp}(np^2,rp)+
\left(\frac{D}{p}\right)p^{k-1}c_{f_\kappa^\sharp}(n,r).\] 

We now compute the sum over all $Q\in  \mathcal{P}_{D_0,r_0}^{(D,r)}(N,p)$ 
of the second summand in Equation \eqref{p-stab}. 
View $Q\in \mathcal{P}_{D_0,r_0}^{(D,r)}(N,p)$ as an element 
of $\mathcal{R}_{D_0,r_0}^{(Dp^2,rp)}(N)$ using Lemma \ref{rep2}. 
For each such $Q$, the form $Q_p$ has discriminant $D_0D$, 
is $p$-primitive (because $p^2\nmid D_0D$), 
and $b\equiv r_0r\pmod{N}$, so it belongs to $\mathcal{P}_{D_0,r_0}^{(D,r)}(N)$. 
We claim that the subset $\mathcal{R}_{D_0,r_0}^{(D,r)}(N)=\{Q_p|Q\in\mathcal{P}_{D_0,r_0}^{(D,r)}(N,p)\}$ 
of $\mathcal{P}_{D_0,r_0}^{(D,r)}(N)$
thus obtained forms a system of representatives for $\mathcal{F}_{D_0,r_0}^{(D,r)}(N)$.
To show this, first note that $\mathcal{P}_{D_0,r_0}^{(D,r)}(N)=\mathcal{F}_{D_0,r_0}^{(D,r)}(N)$: 
since $p^2\nmid D_0D$, an integral quadratic form may have discriminant $D_0D$ only 
if it is $p$-primitive. Moreover, the map $[a,b,c]\mapsto [p^2a,pb,c]$ takes the set 
$\mathcal{P}_{D_0,r_0}^{(D,r)}(N)$ to $\mathcal{P}_{D_0,r_0}^{(Dp^2,rp)}(N)$, 
and is clearly an inverse of the map $[a,b,c]\mapsto [a/p^2,b/p,c]$ defined above, 
thus proving the claim. 
Since $\chi_{D_0}(Q)=\chi_{D_0}(Q_p)$, 
we have 
\[
\sum_{Q\in\mathcal{P}_{D_0,r_0}^{(D,r)}(N,p)}
\chi_{D_0}(Q_p)\cdot  I^-_{f_{\kappa}^\sharp}\{r_{Q_p}\rightarrow s_{Q_p}\}\left(Q_p^{k-1}\right) =
\sum_{Q\in\mathcal{F}_{D_0,r_0}^{(D,r)}(N)}
\chi_{D_0}(Q)\cdot  I^-_{f_{\kappa}^\sharp}\{r_{Q}\rightarrow s_{Q}\}\left(Q^{k-1}\right).\]
The term on the right hand side is $c_{f_\kappa^\sharp}(n,r)$. 

Putting everything together, we get the formula 
\[c_{n,r}(\tilde\kappa)= \lambda(k)\cdot
\left(
a_p(\kappa)\cdot c_{f_\kappa^\sharp}(n,r)-
p^{k-1}\left(\left(\frac{D_0}{p}\right)+\left(\frac{D}{p}\right)\right)c_{f_\kappa^\sharp}(n,r)
-\frac{p^{2k-2}}{a_p(\kappa)}c_{f_\kappa^\sharp}(n,r)
\right)\]
from which the result follows. 
\end{proof}

\begin{rmk} The results described in this section can be used to link Stark-Heegner points (\cite{Dar}, \cite{BD-SH}; see also \cite{LV-IMRN} and \cite{LMH})
to $p$-adic derivatives of Jacobi-Fourier coefficients of $p$-adic families of Jacobi forms, as done in \cite{DT} and \cite{LM} for $p$-adic families of half-integral weight modular forms. Details are left to the interested reader.\end{rmk}

\section{A pointer to \cite{LN-GKZ}} \label{final sec}
In \cite{LN-GKZ}, we consider the compatibility of the Big Heegner points construction of B. Howard of \cite{Ho2} with the formation of $\Lambda$-adic Jacobi forms, in the spirit of the classical Gross-Kohnen-Zagier theorem cf. \cite{GKZ} which links Heegner points (or more generally, Heegner cycles in higher weight) with the theta lift of a weight 2 modular form in the space of Jacobi forms. More precisely, 
choose a index $N$ pair $(D_0,r_0)$ such that $p$ is split in $D_0$; then the generating series naturally constructed from
Big Heegner points and the index $N$ pair $(D_0,n_0)$ 
is the formal series 
\[ \mathbb{H}_{D_0,r_0}=\sum_{n,r}\alpha_{n,r}q^n\zeta^r\] 
in $\mathcal{R}\pwseries{q,\zeta}$, 
where $\alpha_{n,r}$ are more precisely
defined as follows. Let $\Sel(\Q,\mathbb{T}^\dagger)$ be the 
Greenberg Selmer group of the self-dual twist $\T^\dagger$ 
of Hida's Big Galois representation $\T$. Under the assumptions of \cite{LN-GKZ}, this is an $\mathcal{R}$-module of rank $1$. Fix a generator $\mathfrak{Z}$ 
of the torsion-free quotient of $\Sel(\Q,\T^\dagger)$, and let $\mathcal{K}:=\mathrm{Frac}(\mathcal{R})$. The $\mathcal{K}$-vector 
space $\Sel_\mathcal{K}(\Q,\T^\dagger):=\Sel(\Q,\T^\dagger)\otimes_\mathcal{R}\mathcal{K}$ is one-dimensional, and therefore 
in $\Sel_\mathcal{K}(\Q,\T^\dagger)$ we may write 
$\mathfrak{Z}_{D,r}^\mathrm{How}=\alpha_{n,r}\cdot\mathfrak{Z}$, where $\mathfrak{Z}_{D,r}^\mathrm{How}\in\Sel(\Q,\T^\dagger)$ are Howard's Big Heegner points, and $\alpha_{n,r}$ are a priori elements in $\mathcal{K}$; moreover, 
by our choice of $\mathfrak{Z}$, we actually have that $\alpha_{n,r}\in\mathcal{R}$. 
{\it Loc. cit.} links the coefficients $\alpha_{n,r}$ explicitly with Fourier coefficients $c_{f_\kappa^\sharp}(n,r)$ associated to $f_{\kappa}^{\sharp}$ as in Theorem \ref{new forms}, at least in small connected neighborhoods of even weights (see \cite[Rmk 5.3]{LN-GKZ} for a discussion of our perspective over the whole weight space).
We briefly point out why the Fourier coefficients $c_{n,r}(\tilde\kappa)$ associated to $f_{\kappa}$ indeed vary $p$-adically, even though $f_{\kappa}^{\sharp}$ itself does not. Indeed, for $2k>2$, the term 
\[ \left(1-
\frac{p^{k-1}\left(\left(\frac{D_0}{p}\right)+\left(\frac{D}{p}\right)\right)}{a_p(\kappa)}
-\frac{p^{2k-2}}{a_p^2(\kappa)}
\right)\] of Theorem \ref{new forms} is always a $p$-adic unit. It then follows that the coefficients $c_{n,r}(\tilde\kappa)$ and $c_{f_\kappa^\sharp}(n,r)$ are equal, up to non-zero fudge factors not depending on $n$ and $r$. Upon dividing by the expression for $c_{n_0,r_0}(\tilde\kappa)$ on both sides, these fudge factors therefore entirely disappear. The results on $p$-adic variation contained in \cite{LN-GKZ} can therefore be expressed in terms of $c_{n,r}(\tilde\kappa)$ instead of $c_{f_\kappa^\sharp}(n,r)$, although in a slightly more restrictive way.

\bibliographystyle{amsalpha}
\bibliography{references}

\def\cprime{$'$}
\providecommand{\bysame}{\leavevmode\hbox to3em{\hrulefill}\thinspace}
\providecommand{\MR}{\relax\ifhmode\unskip\space\fi MR }
\providecommand{\MRhref}[2]{%
  \href{http://www.ams.org/mathscinet-getitem?mr=#1}{#2}
}
\providecommand{\href}[2]{#2}
\begin{thebibliography}{MRV93}

\bibitem[BD09a]{BD-SH}
M.~Bertolini and H.~Darmon, \emph{The rationality of {S}tark-{H}eegner points
  over genus fields of real quadratic fieldseegner points over genus fields of
  real quadratic fields}, Anna. Math. \textbf{170} (2009), 343--369.

\bibitem[BD09b]{BD}
Massimo Bertolini and Henri Darmon, \emph{The rationality of {S}tark-{H}eegner
  points over genus fields of real quadratic fields}, Ann. of Math. (2)
  \textbf{170} (2009), no.~1, 343--370. \MR{2521118}

\bibitem[Boy15]{Boylan}
Hatice Boylan, \emph{Jacobi forms, finite quadratic modules and {W}eil
  representations over number fields}, Lecture Notes in Mathematics, vol. 2130,
  Springer, Cham, 2015, With a foreword by Nils-Peter Skoruppa. \MR{3309829}

\bibitem[Cana]{CandeloriAlg}
Luca Candelori, \emph{The {A}lgebraic {F}unctional {E}quation of {R}iemann's
  {T}heta {F}unction}, to appear at Ann. Inst. Fourier.

\bibitem[Canb]{CandeloriTheta}
\bysame, \emph{The {T}ransformation {L}aws of {A}lgebraic {T}heta {F}unctions},
  preprint arXiv:1609.04486.

\bibitem[Can14]{CandeloriThesis}
\bysame, \emph{Metaplectic stacks and vector-valuedmodular forms of
  half-integral weight}, Ph.D. Thesis (2014).

\bibitem[Dar01]{Dar}
H.~Darmon, \emph{Integration on $\mathcal{H}_p\times\mathcal{H}$ and arithmetic
  applications}, Ann. Math. \textbf{154} (2001), 589--639.

\bibitem[DR73]{DeligneRapoport}
P.~Deligne and M.~Rapoport, \emph{Les sch\'{e}mas de modules de courbes
  elliptiques}, 143--316. Lecture Notes in Math., Vol. 349. \MR{0337993}

\bibitem[DT08]{DT}
H.~Darmon and G.~Tornar\`'ia, \emph{Stark--heegner points and the shimura
  correspondence}, Compositio Mathematica \textbf{144} (2008), no.~5,
  1155--1175.

\bibitem[EZ85]{EZ}
Martin Eichler and Don Zagier, \emph{The theory of {J}acobi forms}, Progress in
  Mathematics, vol.~55, Birkh\"auser Boston, Inc., Boston, MA, 1985.
  \MR{781735}

\bibitem[GKZ87]{GKZ}
B.~Gross, W.~Kohnen, and D.~Zagier, \emph{Heegner points and derivatives of
  {$L$}-series. {II}}, Math. Ann. \textbf{278} (1987), no.~1-4, 497--562.
  \MR{909238}

\bibitem[GS93]{GS}
Ralph Greenberg and Glenn Stevens, \emph{{$p$}-adic {$L$}-functions and
  {$p$}-adic periods of modular forms}, Invent. Math. \textbf{111} (1993),
  no.~2, 407--447. \MR{1198816}

\bibitem[Gue94]{Gu3}
P.I. Guerzhoy, \emph{An approach to the {$p$}-adic theory of {J}acobi forms},
  Internat. Math. Res. Notices (1994), no.~1, 31--39. \MR{1255251}

\bibitem[Gue95]{Gu2}
\bysame, \emph{Jacobi-{E}isenstein series and {$p$}-adic interpolation of
  symmetric squares of cusp forms}, Ann. Inst. Fourier (Grenoble) \textbf{45}
  (1995), no.~3, 605--624. \MR{1340946}

\bibitem[Gue00a]{GuerzhoyCrelle}
P.~Guerzhoy, \emph{On {$p$}-adic families of {S}iegel cusp forms in the
  {M}aa\ss {S}pezialschar}, J. Reine Angew. Math. \textbf{523} (2000),
  103--112. \MR{1762957}

\bibitem[Gue00b]{Gu1}
P.I. Guerzhoy, \emph{Jacobi forms and a {$p$}-adic {$L$}-function in two
  variables}, Fundam. Prikl. Mat. \textbf{6} (2000), no.~4, 1007--1021.
  \MR{1813009}

\bibitem[Hid86a]{Hida85}
Haruzo Hida, \emph{Galois representations into {${\rm GL}_2({\bf Z}_p[[X]])$}
  attached to ordinary cusp forms}, Invent. Math. \textbf{85} (1986), no.~3,
  545--613. \MR{848685}

\bibitem[Hid86b]{Hida86}
\bysame, \emph{Iwasawa modules attached to congruences of cusp forms}, Ann.
  Sci. \'Ecole Norm. Sup. (4) \textbf{19} (1986), no.~2, 231--273. \MR{868300}

\bibitem[Hid88]{Hida88}
\bysame, \emph{On {$p$}-adic {H}ecke algebras for {${\rm GL}_2$} over totally
  real fields}, Ann. of Math. (2) \textbf{128} (1988), no.~2, 295--384.
  \MR{960949}

\bibitem[Hid93]{HidaEisen}
\bysame, \emph{Elementary theory of {$L$}-functions and {E}isenstein series},
  London Mathematical Society Student Texts, vol.~26, Cambridge University
  Press, Cambridge, 1993. \MR{1216135}

\bibitem[Hid95]{HarHalf}
\bysame, \emph{On {$\Lambda$}-adic forms of half integral weight for {${\rm
  SL}(2)_{/{\bf Q}}$}}, Number theory ({P}aris, 1992--1993), London Math. Soc.
  Lecture Note Ser., vol. 215, Cambridge Univ. Press, Cambridge, 1995,
  pp.~139--166. \MR{1345178}

\bibitem[How07]{Ho2}
Benjamin Howard, \emph{Variation of {H}eegner points in {H}ida families},
  Invent. Math. \textbf{167} (2007), no.~1, 91--128. \MR{2264805}

\bibitem[Ibu12]{Ibu}
Tomoyoshi Ibukiyama, \emph{Saito-{K}urokawa liftings of level {$N$} and
  practical construction of {J}acobi forms}, Kyoto J. Math. \textbf{52} (2012),
  no.~1, 141--178. \MR{2892771}

\bibitem[Kra91]{Kramer91}
J.~Kramer, \emph{A geometrical approach to the theory of {J}acobi forms},
  Compositio Math. \textbf{79} (1991), no.~1, 1--19.

\bibitem[Kra95]{Kramer95}
\bysame, \emph{An arithmetic theory of {J}acobi forms in higher dimensions}, J.
  Reine Angew. Math. \textbf{458} (1995), 157--182. \MR{1310957}

\bibitem[LM18]{LM}
Matteo Longo and Zhengyu Mao, \emph{Kohnen's formula and a conjecture of
  {D}armon and {T}ornar\'ia}, Trans. Amer. Math. Soc. (2018).

\bibitem[LMH20]{LMH}
Matteo Longo, Kimball Martin, and Yan Hu, \emph{Rationality of darmon points
  over genus fields of non-maximal orders}, Annales math{\'e}matiques du
  Qu{\'e}bec \textbf{44} (2020), no.~1, 173--195.

\bibitem[LN13]{LN-Doc}
Matteo Longo and Marc-Hubert Nicole, \emph{The {$\Lambda$}-adic
  {S}himura-{S}hintani-{W}aldspurger correspondence}, Doc. Math. \textbf{18}
  (2013), 1--21. \MR{3035767}

\bibitem[LN19]{LN-GKZ}
\bysame, \emph{On the $p$-adic variation of the {G}ross-{K}ohnen-{Z}agier
  {T}heorem}, Forum Math. \textbf{31} (2019), no.~4, 1069--1084.

\bibitem[LV11]{LV2}
Matteo Longo and Stefano Vigni, \emph{Quaternion algebras, {H}eegner points and
  the arithmetic of {H}ida families}, Manuscripta Math. \textbf{135} (2011),
  no.~3-4, 273--328. \MR{2813438}

\bibitem[LV14]{LV-IMRN}
\bysame, \emph{The rationality of quaternionic {D}armon points over genus
  fields of real quadratic fields}, Int. Math. Res. Not. IMRN (2014), no.~13,
  3632--3691. \MR{3229764}

\bibitem[MB90]{MBCompositio}
Laurent Moret-Bailly, \emph{Sur l'\'{e}quation fonctionnelle de la fonction
  th\^{e}ta de {R}iemann}, Compositio Math. \textbf{75} (1990), no.~2,
  203--217. \MR{1065206}

\bibitem[MR00]{MR}
M.~Manickam and B.~Ramakrishnan, \emph{On {S}himura, {S}hintani and
  {E}ichler-{Z}agier correspondences}, Trans. Amer. Math. Soc. \textbf{352}
  (2000), no.~6, 2601--2617. \MR{1637086}

\bibitem[MRV93]{MRVSK}
M.~Manickam, B.~Ramakrishnan, and T.~C. Vasudevan, \emph{On {S}aito-{K}urokawa
  descent for congruence subgroups}, Manuscripta Math. \textbf{81} (1993),
  no.~1-2, 161--182. \MR{1247596}

\bibitem[Nek06]{Nek}
Jan Nekov{\'a}{\v{r}}, \emph{Selmer complexes}, Ast\'erisque (2006), no.~310,
  viii+559. \MR{2333680}

\bibitem[Ram06]{Ramsey2006}
Nick Ramsey, \emph{Geometric and {$p$}-adic modular forms of half-integral
  weight}, Ann. Inst. Fourier (Grenoble) \textbf{56} (2006), no.~3, 599--624.
  \MR{2244225}

\bibitem[Ste94]{St}
Glenn Stevens, \emph{{$\Lambda$}-adic modular forms of half-integral weight and
  a {$\Lambda$}-adic {S}hintani lifting}, Arithmetic geometry ({T}empe, {AZ},
  1993), Contemp. Math., vol. 174, Amer. Math. Soc., Providence, RI, 1994,
  pp.~129--151. \MR{1299739}

\end{thebibliography}
\end{document}